\documentclass{siamltex}
\usepackage{amssymb}
\usepackage{lmodern}
\usepackage{bm}
\usepackage{graphicx}
\usepackage{graphics}
\usepackage{caption2}

\usepackage{psfrag}
\usepackage{amsmath}
\usepackage{lineno}
\usepackage{amscd}
\usepackage{amsfonts}
\usepackage{float}
\usepackage{latexsym}
\usepackage{lscape}
\usepackage{algorithm,algorithmic}
\usepackage{comment}
\usepackage{color}
\makeatletter
\newcommand*{\rom}[1]{\expandafter\@slowromancap\romannumeral #1@}
\makeatother

\usepackage[marginal]{footmisc}

\newtheorem{remark}{Remark}[section]

\newcommand{\ds}{\displaystyle}
\newcommand{\f}{\frac}

\newcommand{\om}{\Omega}
\newcommand{\p}{\partial}

\begin{document}
	%\linenumbers
\title{Addressing complex boundary conditions of miscible flow and transport in two and three dimensions with application to optimal control}

\author{
Yiqun Li\thanks{Department of Mathematics, University of South Carolina, Columbia, SC 29208, USA}
\and
Hong Wang\thanks{Department of Mathematics, University of South Carolina, Columbia, SC 29208, USA}
\and
Xiangcheng Zheng\thanks{School of Mathematics, Shandong University, Jinan 250100, China (Corresponding author, email: xzheng@sdu.edu.cn)}
}

	\maketitle
\begin{abstract}
We investigate complex boundary conditions of the miscible displacement system in two and three space dimensions with the commonly-used Bear-Scheidegger diffusion-dispersion tensor,  which describes, e.g., the porous medium flow processes in petroleum reservoir simulation or  groundwater contaminant transport.  Specifically, we incorporate the no-flux boundary condition for the Darcy velocity to prove that the general no-flux boundary condition for the transport equation is equivalent to the normal derivative boundary condition of the concentration,  based on which we further prove several complex boundary conditions by the Bear-Scheidegger tensor and its derivative. The derived boundary conditions not only provide new insights and distinct properties of the Bear-Scheidegger diffusion-dispersion tensor, but accommodate the coupling and the nonlinearity of the miscible displacement system and the Bear-Scheidegger tensor in deriving the first-order optimality condition of the corresponding optimal control problem for practical application.
	\end{abstract}

	
	
	\pagestyle{myheadings}
	\thispagestyle{plain}
	\markboth{}{Complex boundary conditions of  miscible displacement system}
\section{Introduction}\label{S:intro}
\subsection{Problem formulation}
%Optimal control has demonstrated widespread applications in various fields and thus has been extensively studied in the literature \cite{Gon,Gun,Herzog,ItoKun,NeiTib}.
The incompressible miscible flow has attracted increasing attentions since it is shown to provide very competitive description of porous medium flow processes in petroleum reservoir simulation and  groundwater contaminant transport \cite{GanVon,MxCar,WanSINUM,WanSISC}.
Let $\Omega \subset \mathbb{R}^d (d \ge 2)$   be a $d$-dimensional bounded domain with the boundary $\p \Omega  $. Let $c(\bm x, t) $ represent the concentration of the injected fluid, which consists of surfactant or other chemicals that are mixed with water  to formulate a fully miscible fluid phase  \cite{Bal,Ewi84,Fen1} and let $p(\bm x, t)$ and $\bm v(\bm x, t)=(v_1,\ldots, v_d)$ be the pressure and Darcy velocity of the fluid mixture, respectively.

We consider the following incompressible miscible flow model has been widely used in describing  two-component,
incompressible, miscible displacement    \cite{Bear72}
\begin{equation}\label{VtFDEs}\begin{array}{cl}
\hspace{-0.15in}\ds \phi(\bm x)\p_t c   -  \nabla \cdot (\bm D(\bm x, \bm v) \nabla c) + \nabla \cdot (c \bm v)= \bar c q,~&\hspace{-0.075in} \mathrm{in} ~\Omega \times (0,T]; \\ [0.05in]
\ds \bm v = -K(\bm x) \nabla p, \quad  \ds \nabla \cdot \bm v = q, ~&\hspace{-0.075in} \mathrm{in} ~\Omega \times [0,T]; \\[0.05in]
	c(\bm x,0) = 0, \quad &\hspace{-0.075in} \mathrm{on}~\om; \\[0.05in]
	\bm v(\bm x, t) \cdot \bm n(\bm x) = (\bm D \nabla c)(\bm x,t) \cdot \bm n (\bm x) = 0,  &\hspace{-0.075in} \mathrm{on} ~ \p \om \times [0,T].
\end{array}\end{equation}
Here $\phi(\bm x)>0$ and $K(\bm x) >0$ represent the porosity and   permeability of the medium, respectively, and   $\nabla := (\p/\p x_1, \ldots, \p/\p x_d)^\top$ is the gradient operator.   $\bm D(\bm x, \bm v) = \big(D_{i,j}(\bm x, \bm v)\big)_{i,j=1}^d$ is the well-known Bear–Scheidegger diffusion-dispersion tensor defined by \cite{CheEwi,LiSunSINUM2015,SunWhe,WanSINUM,WanSISC}
\begin{equation}\label{D}
\ds \bm D(\bm x, \bm v) : = \phi(\bm x)d_m \bm I + d_t |\bm v| \bm I + (d_l-d_t) (v_i v_j)_{i,j=1}^d/|\bm v|
\end{equation}
 with $d_m$, $d_t$, $d_l>0$ being the molecular diffusion and the transverse and longitudinal dispersivities, respectively,
and $\bm I$ is the identity tensor. The vector $\bm n$ denotes the unit normal vector outward to $\p \Omega$, $q$  denotes the
external source/sink term, and   $\bar c (\bm x, t)$ is specified at sources and  $\bar c (\bm x, t) = c (\bm x, t)$ at sinks.

 For the displacement processes occurring in petroleum reservoir simulation, the Darcy's law in the second equation in \eqref{VtFDEs} typically depends on the viscosity of the  fluid mixture $\mu$, which further incorporates the gravitational effects  for the three-dimensional porous medium reservoir.
 % such that the second equation in \eqref{VtFDEs} is augmented as \cite{WanSISC,WanZhao}
%\begin{equation}\label{Darcy3d}
%\ds \bm v = -\f{K(\bm x)}{\mu(c)} \big(\nabla p - \rho g \nabla d\big),
%\end{equation}
% where $\rho$ is the density of the fluid mixture, $g$ is the magnitude of gravitational acceleration, and $d(\bm x)$  is the reservoir depth.
 Here we omit the gravitational and viscosity effects  and instead consider the model \eqref{VtFDEs} for  simplicity of presentation.
\subsection{Main contribution}
Miscible displacement system with Bear-Scheidegger tensor has been extensively studied in the literature with sophisticated mathematical or numerical analysis progresses \cite{Arb,Fen,LiSun,LiSunSINUM2015,RivWal,WanSINUM}. From the application point of view, there also exist some works on the optimal control of porous media flow in groundwater contaminant transport or in oil reservoir simulation \cite{ChaCao,KumRui}. % In optimal control of  immiscible displacement, the scalar diffusion coefficients depending only on the concentration are adopted in the governing equations (\ref{VtFDEs}).
 For more practical Bear-Scheidegger diffusion-dispersion tensor (\ref{D}), the corresponding optimal control remains untreated due to coupling and the nonlinearity of the system and the Bear-Scheidegger tensor. Specifically, deriving the first-order optimality condition requires several complex boundary conditions that are not encountered in the forward problem \eqref{VtFDEs} and even in the optimal control of the displacement system with scalar diffusion coefficients  \cite{ChaCao,KumRui}, and thus necessitate substantial investigations.

In this work, we address the aforementioned issue to prove novel boundary conditions of (\ref{VtFDEs}) for both two- and three-dimensional problems by employing properties of the Bear-Scheidegger tensor and its derivative, which provide new insights and distinct properties of the Bear-Scheidegger diffusion-dispersion tensor.  Specifically, we incorporate the no-flux boundary condition for the Darcy velocity to prove that the general no-flux boundary condition for the transport equation is equivalent to the normal derivative boundary condition of the concentration, based on which we further prove several complex boundary conditions by the Bear-Scheidegger tensor and its derivative.  We then apply the proved boundary conditions  to derive the first-order optimality condition of the optimal control problem by carefully designing the adjoint system, which demonstrates the applicability of the proposed boundary conditions.
Specifically, we prove complex boundary conditions for the two-dimensional case in \S \ref{Sect:Lem}, and then apply them to derive the first-order optimality condition of the optimal control of the miscible displacement system (\ref{Model}) in \S \ref{Sect:Cond}. We finally prove analogous results for the three-dimensional case by more technical treatments in \S \ref{Sect:Lem3d}.

\subsection{Preliminaries and assumptions}\label{notation}

  Let $L^q(\om)$ with $1 \le q \le \infty$ be the Banach space of $q$th power Lebesgue integrable functions on $\om$.
   For a positive integer $m$, let  $ W^{m, q}(\Omega)$ be the Sobolev space of $L^q$ functions with $m$th weakly derivatives in $L^q(\om)$.
Let $H^m(\om) = W^{m, 2}(\om)$ and $H_0^m(\Omega)$ be the completion of $C_0^\infty(\Omega)$, the space of infinitely differentiable functions with compact support in $\Omega$, in $H^m(\Omega)$ \cite{AdaFou,Eva}.
For a Banach space $\mathcal{X}$, let $W^{m, q}(0,T; \mathcal{X})$ be the space of functions in $W^{m, q}(0,T)$ with respect to $\|\cdot\|_{\mathcal {X}}$. All spaces are equipped with standard norms \cite{AdaFou,Eva}.
% We may also drop the notation $\Omega$ in the inner product and Sobolev spaces and norms, and write $W^{m,q}(X)$ for $W^{m,q}(0, T; X)$ when no confusion occurs, e.g. we write $L^2(L^2)$ instead of $L^2(0,T;L^2(\Omega))$ for simplicity.
 % In subsequent sections, we use $Q$ and $M$ to denote positive constants where $Q$ is independent from the quantities of interest but may assume different values at different occurrences.

Recall that $q$ denotes the external source/sink term and  we impose $\int_\Omega q(\bm x, t) d \bm x =0$ for a.e. $t \in [0, T]$, which states that for an incompressible flow with an impermeable boundary, the amount of injected fluid should be equal to the amount that is produced.
%Similarly, we impose $\int_\Omega (p - p_d)(\bm x, t) d \bm x =0$ for  a.e. $t \in [0, T]$ for the elliptic equation with respect to $\psi_p$ with the no-flow boundary condition in  the adjoint equations  \eqref{AdjEq}.
Furthermore, the forward problem \eqref{VtFDEs}  with  no-flow boundary conditions can only determine the pressure $p$ up to an additive constant. To  ensure the  uniqueness of  $p$, we  enforce   $\int_\Omega p(\bm x, t) d \bm x =0$ for  a.e. $t \in [0, T]$.
Based on these discussions, we follow \cite{LiSunSINUM2013,WanSi} to make the following {\it assumptions}:
\begin{itemize}
\item[(a)]  $\p_{v_i}  D_{i, j}$, $\p_{v_j}  D_{i, j} \in L^\infty(L^\infty)$ for $ 1 \le i, j \le d$.
\item[(b)]   $ \phi >\phi_* $ for some $\phi_*>0$.
\item[(c)] $\int_{\Omega} q\, d\bm x = \int_{\Omega} p\, d\bm x =0$ for a.e. $t \in [0, T]$.
\end{itemize}

\section{Complex boundary conditions in 2D}\label{Sect:Lem}
We prove complex boundary conditions for the two-dimensional case.
\begin{lemma}\label{Lem:D} For $\bm D(\bm x, \bm v)$ defined in \eqref{D} with $d =2$, $\bm v(\bm x, t) \cdot \bm n(\bm x)=0$ on $ \p \om \times [0,T]$ and some  scalar function $g$, we have
\begin{equation}\begin{array}{l}\label{LemD:eq1}
  \ds (\bm D \nabla g) \cdot \bm n(\bm x) =0 \Longleftrightarrow   \nabla g \cdot \bm n(\bm x) =0 ~~ \mbox{on}~~  \p \om \times [0,T].
  \end{array}
\end{equation}
\end{lemma}
\begin{proof}
By \eqref{D}, we have
\begin{equation}\begin{array}{cl}\label{LemD:Dg}
\hspace{-0.15in}\ds (\bm D \nabla g) \cdot \bm n(\bm x) = \big(\phi  + d_t |\bm v|\big) \nabla g \cdot \bm n(\bm x) +  \f{(d_l-d_t)}{|\bm v| } \big((v_i v_j)_{i,j=1}^2 \nabla g \big)\cdot \bm n(\bm x),
\end{array}
\end{equation}
where, by  $\bm n(\bm x) : = (\cos\theta, \sin\theta)$ for some  $\theta$ depending on $\bm x \in \p \om$, the last term could be further evaluated on $\p\om$ as
\begin{equation}\begin{array}{cl}\label{LemD:Dg:e1}
\hspace{-0.15in}\ds \big((v_i v_j)_{i,j=1}^2 \nabla g \big)\cdot \bm n(\bm x) & \ds \hspace{-0.1in}=
\begin{bmatrix}
  v_1^2 \p_x g + v_1 v_2 \p_y g  \\[0.1in]
  v_1 v_2 \p_x g + v_2^2 \p_y g
\end{bmatrix} \cdot \bm n(\bm x)\\[0.2in]
 & \ds \hspace{-0.8in}=\p_x g\big(v_1^2 \cos\theta + v_1 v_2 \sin\theta\big)  + \p_y g\big(v_1 v_2 \cos\theta + v_2^2 \sin\theta\big)=0.
\end{array}
\end{equation}
Here we have used the fact that
\begin{equation}\label{bd:v}
\ds v_1 \cos\theta + v_2 \sin\theta=0
\end{equation}
 derived from $\bm v \cdot \bm n(\bm x)=0$ on $ \p \om \times [0,T]$. Thus, we obtain from \eqref{LemD:Dg}--\eqref{LemD:Dg:e1} that $(\bm D \nabla g) \cdot \bm n(\bm x) = \big(\phi  + d_t |\bm v|\big) \nabla g \cdot \bm n(\bm x) $ on $\p \om \times [0,T]$, which, together with $\phi  + d_t |\bm v|>\phi_*>0$,  completes the proof.
\end{proof}

\begin{lemma}\label{Lem1} Under the boundary conditions
\begin{equation}\label{Lem1:bd}
 \bm v(\bm x, t) \cdot \bm n(\bm x)= \nabla p(\bm x, t) \cdot \bm n(\bm x)= (\bm D \nabla c) \cdot \bm n(\bm x) =0 ~~ \mbox{on}~~  \p \om \times [0,T],
\end{equation}
we have $\ds \big[ (\bm E (\bm x, \bm v) \otimes \nabla p) \nabla c\big] \cdot \bm n(\bm x) =0$  \mbox{on} $\p \om \times [0,T]$, where $\bm E =\big(\p_{\bm v}D_{i,j}\big)_{i,j=1}^2$  is a 2-by-2 block matrix with entries given by
\begin{equation}\begin{array}{l}\label{DhatE}
\hspace{-0.15in}\ds (\p_{v_1}D_{1,1},\p_{v_2}D_{1,1}) \ds =\p_{\bm v} D_{1,1}   = \p_{\bm v}\Big( dt |\bm v| +  (d_l-d_t){v_1^2}/{|\bm v|}\Big)\\[0.1in]
 \ds \qquad \qquad \qquad \qquad \qquad \quad  =   dt \Big(\f{v_1}{|\bm v|}, \f{v_2}{|\bm v|}\Big) \!+\! (d_l-d_t)  \Big(\f{v_1(v_1^2 + 2 v_2^2)}{|\bm v|^3}, \f{-v_1^2 v_2}{|\bm v|^3}\Big), \\[0.15in]
%& \hspace{-0.15in} \ds = \Big( \f{d_lv_1^3 +(2d_l -d_t) v_1v_2^2}{|\bm v|^3}, \f{d_t v_2^3 + (2d_t-d_l) v_1^2v_2}{|\bm v|^3}\Big) ,\\[0.15in]
\hspace{-0.15in} \ds \ds (\p_{v_1}D_{1,2},\p_{v_2}D_{1,2}) = \p_{\bm v}  D_{1,2}  \ds  = (d_l-d_t) \p_{\bm v}\Big(  {v_1 v_2}/{|\bm v|}\Big) = (d_l-d_t)  \Big(\f{v_2^3}{|\bm v|^3}, \f{v_1^3}{|\bm v|^3}\Big),\\[0.1in]
\hspace{-0.15in}  \ds (\p_{v_1}D_{2,1},\p_{v_2}D_{2,1})\ds  =  (\p_{v_1}D_{1,2},\p_{v_2}D_{1,2}),\\[0.1in]
\hspace{-0.15in} \ds \ds (\p_{v_1}D_{2,2},\p_{v_2}D_{2,2}) \ds  = \p_{\bm v}  D_{2,2}   = \p_{\bm v}\Big( dt |\bm v| +  (d_l-d_t){v_2^2}/{|\bm v|}\Big)\\[0.1in]
 \ds \qquad \qquad \qquad \qquad \qquad \quad  =   dt \Big(\f{v_1}{|\bm v|}, \f{v_2}{|\bm v|}\Big)\! +\! (d_l-d_t)  \Big(\f{-v_1 v_2^2 }{|\bm v|^3}, \f{v_2(2 v_1^2 +  v_2^2)}{|\bm v|^3}\Big)%\\[0.15in]
%& \hspace{-0.15in}\ds = \Big(\f{d_t v_1^3 + (2d_t-d_l) v_1v_2^2}{|\bm v|^3}, \f{d_lv_2^3 +(2d_l -d_t) v_1^2v_2}{|\bm v|^3}\Big).
\end{array}
\end{equation}
and $\otimes$ is interpreted as a Kronecker type product of a 2-by-2 block matrix and a column vector in $\mathbb R^{2\times 1}$ such that
\begin{equation}\begin{array}{l}\label{Def:kron}
  \ds {\bm E} (\bm x, \bm v)\otimes  \nabla c :=
  \begin{bmatrix}
\big( \p_{ v_1}  D_{1,1}, \p_{ v_2}  D_{1,1}\big)\nabla c   &\big( \p_{ v_1}  D_{1,2}, \p_{v_2}  D_{1,2}\big)\nabla c    \\[0.1in]
\big( \p_{v_1}  D_{2,1}, \p_{ v_2}  D_{2,1}\big) \nabla c    & \big( \p_{ v_1}  D_{2,2}, \p_{ v_2}  D_{2,2}\big)\nabla c
\end{bmatrix}\in \mathbb R^{2\times 2}.
\end{array}
\end{equation}
 \end{lemma}
\begin{proof}
We split $\p_{\bm v} D_{1,1}$ and $\p_{\bm v} D_{2,2}$ in \eqref{DhatE} as
\begin{equation}\begin{array}{cl}\label{Lem:hatE}
\hspace{-0.15in}\ds \p_{\bm v} D_{1,1} &\hspace{-0.1in}\ds   =   dt \Big(\f{v_1}{|\bm v|}, \f{v_2}{|\bm v|}\Big) + (d_l-d_t)  \Big(\f{v_1(v_1^2 + 2 v_2^2)}{|\bm v|^3}, \f{-v_1^2 v_2}{|\bm v|^3}\Big) = : \bm E_d^t + \bm E_{1,1}^l,\\[0.15in]
%\hspace{-0.15in} \ds \p_{\bm v} \bm D_{1,2}  &\hspace{-0.1in}\ds  = (d_l-d_t)  \Big(\f{v_2^3}{|\bm v|^3}, \f{v_1^3}{|\bm v|^3}\Big),\quad \p_{\bm v} \bm D_{2,1}=  \p_{\bm v} \bm D_{1,2},\\[0.1in]
\hspace{-0.15in} \ds \p_{\bm v}  D_{2,2} &\hspace{-0.1in}\ds  =   dt \Big(\f{v_1}{|\bm v|}, \f{v_2}{|\bm v|}\Big) + (d_l-d_t)  \Big(\f{-v_1 v_2^2 }{|\bm v|^3}, \f{v_2(2 v_1^2 +  v_2^2)}{|\bm v|^3}\Big) = : \bm E_d^t + \bm E_{2,2}^l,
\end{array}
\end{equation}
and thus to split $\bm E =: \bm E^t + \bm E^l  $ as the sum of two 2-by-2  block  matrices
\begin{equation}\begin{array}{cl}\label{Lem:Matrix}
\ds \bm E^t: = \diag(\bm E^t_d, \bm E^t_d), \quad \bm E^l = \big(\bm E_{i,j}^l\big)_{i,j=1}^2: =
\begin{bmatrix}
  \bm E_{1,1}^l & \p_{\bm v}  D_{1,2}  \\[0.1in]
  \p_{\bm v}  D_{2,1} & \bm E_{2,2}^l
\end{bmatrix}
\end{array}.
\end{equation}
We first show that $\big[ (\bm E^t \otimes \nabla p) \nabla c\big ]\cdot \bm n(\bm x) =0$ on $\p \om \times [0,T]$. By Lemma \ref{Lem:D}, $(\bm D \nabla c) \cdot \bm n(\bm x) =0 $ on $ \p \om \times [0,T]$ implies that $\nabla c \cdot \bm n(\bm x) =0$ on $ \p \om \times [0,T]$. By \eqref{Def:kron} and \eqref{Lem:hatE}--\eqref{Lem:Matrix}, we have
\begin{equation}\begin{array}{l}\label{Lem:Et}
\hspace{-0.175in}\ds  \big[ (\bm E^t \otimes  \nabla p) \nabla c\big ]\cdot \bm n(\bm x)\ds =
\begin{bmatrix}
\bm E^t_d  \nabla p  &0\\
0 & \bm E^t_d  \nabla p
\end{bmatrix}
\nabla c \cdot \bm n(\bm x)
%\begin{bmatrix}
%\big(E^t \cdot \nabla p\big)  &0\\
%0 & \big(E^t \cdot \nabla p\big)
%\end{bmatrix}
%\begin{bmatrix}
%\p_x c  \\
%\p_y c
%\end{bmatrix} \cdot \bm n(\bm x) \\[0.175in]
\ds = \big(\bm E^t_d   \nabla p\big) \nabla c \cdot \bm n(\bm x) = 0
\end{array}
\end{equation}
on $ \p \om \times [0,T]$.
By \eqref{Lem:Matrix}, we then combine \eqref{Def:kron} to evaluate $\big[(\bm E^l \otimes\nabla p) \nabla c\big] \cdot \bm n(\bm x) $ on $ \p \om \times [0,T]$ as follows
\begin{equation}\begin{array}{cl}\label{Lem:El}
\ds  \big[(\bm E^l \otimes \nabla p) \nabla c\big] \cdot \bm n(\bm x)  & \hspace{-0.1in}\ds =
\begin{bmatrix}
\bm E^l_{1,1}  \nabla p  &\p_{\bm v}  D_{1,2} \nabla p  \\[0.1in]
\p_{\bm v} D_{2,1} \nabla p  & \bm E^l_{2,2}  \nabla p
\end{bmatrix}
\nabla c \cdot \bm n(\bm x)\\[0.25in]
& \hspace{-0.1in}\ds
 =
 \begin{bmatrix}
\big(\bm E^l_{1,1}  \nabla p\big)\p_x c  + \big(\p_{\bm v}  D_{1,2}  \nabla p\big)\p_y c  \\[0.1in]
\big(\p_{\bm v}  D_{2,1}   \nabla p\big)\p_x c  + \big( \bm E^l_{2,2}  \nabla p\big)\p_y c
\end{bmatrix} \cdot \bm n(\bm x)\\[0.25in]
& \hspace{-0.1in}\ds = \p_x c \Big[ \cos\theta\big(\bm E^l_{1,1} \nabla p\big) + \sin\theta  \big(\p_{\bm v}  D_{2,1}   \nabla p\big)\Big] \\[0.1in]
& \hspace{-0.1in}\ds \qquad + \p_y c \Big[ \cos\theta\big(\p_{\bm v}  D_{1,2}  \nabla p\big) + \sin\theta  \big(\bm E^l_{2,2}  \nabla p\big)\Big] \\[0.1in]
& \hspace{-0.1in}\ds = : \f{d_l-d_t}{|\bm v|^3}\big( Q_1 \p_x c  + Q_2 \p_y c\big).
\end{array}
\end{equation}
Here $Q_1$ could be computed by incorporating \eqref{bd:v}, \eqref{DhatE}, and \eqref{Lem:hatE}  as follows
\begin{equation}\begin{array}{cl}\label{Lem:Q1}
\hspace{-0.15in}\ds Q_1 %&\hspace{-0.125in} \ds =  \cos\theta\big(\bm E^l_{1,1} \cdot \nabla p\big) + \sin\theta  \big(\p_{\bm v} \bm D_{2,1}  \cdot \nabla p\big) \\[0.1in]
& \hspace{-0.125in} \ds = (\cos\theta v_1) \big[ (v_1^2 + 2v_2^2) \p_x p -v_1 v_2 \p_y p \big]  + \sin\theta\big[v_2^3\p_x p + v_1^3 \p_y p \big] \\[0.1in]
\hspace{-0.15in}& \hspace{-0.125in} \ds = (-\sin\theta v_2) \big[ (v_1^2 + 2v_2^2) \p_x p -v_1 v_2 \p_y p \big]  + \sin\theta\big[v_2^3\p_x p + v_1^3 \p_y p \big] \\[0.1in]
\hspace{-0.15in}& \hspace{-0.125in} \ds = \sin\theta\big[ - v_2 |\bm v|^2 \p_x p + v_1|\bm v|^2 \p_y p\big]  =  (\cos\theta v_1)|\bm v|^2 \p_x p +\sin\theta v_1|\bm v|^2 \p_y p\\[0.1in]
\hspace{-0.15in}& \hspace{-0.125in} = |\bm v|^2 v_1 \big[ \cos\theta \p_x p +  \p_y p \sin\theta\big] =0,
\end{array}
\end{equation}
where the last equality employs $\nabla p \cdot \bm n(\bm x) =0$ on $ \p \om \times [0,T]$.
Similarly, we evaluate $Q_2$ in \eqref{Lem:El} as follows
\begin{equation}\begin{array}{cl}\label{Lem:Q2}
\hspace{-0.15in} \ds Q_2 %&\hspace{-0.1in} \ds =  \cos\theta\big(\p_{\bm v} \bm D_{2,1}\cdot \nabla p\big) + \sin\theta  \big(\bm E^l_{2,2}  \cdot \nabla p\big) \\[0.1in]
& \hspace{-0.1in} \ds = \cos\theta  \big[v_2^3\p_x p + v_1^3 \p_y p \big]   + v_2\sin\theta\big[-v_1v_2  \p_x p  + (2v_1^2 + v_2^2) \p_y p \big] \\[0.1in]
& \hspace{-0.1in} \ds = \cos\theta  \big[v_2^3\p_x p + v_1^3 \p_y p \big] - \cos\theta v_1 \big[ -v_1v_2  \p_x p  + (2v_1^2 + v_2^2) \p_y p  \big] \\[0.1in]
& \hspace{-0.1in} \ds = \cos\theta  \big[v_2 |\bm v|^2 \p_x p - v_1|\bm v|^2  \p_y p\big] = |\bm v|^2 v_2 \big[ \p_x p \cos\theta+  \p_y p\sin\theta\big]=0.
\end{array}
\end{equation}

We incorporate \eqref{Lem:Q1}--\eqref{Lem:Q2} to find that $\big[(\bm E^l \otimes \nabla p) \nabla c\big] \cdot \bm n(\bm x) =0$  on $ \p \om \times [0,T]$, which, together with \eqref{Lem:Matrix} and \eqref{Lem:Et}, yields the conclusion of this lemma.
\end{proof}

For the convenience of analysis, we define a 2-by-2 block matrix $\hat{\bm E }$   expressed as follows
\begin{equation}\begin{array}{cl}\label{hatE}
  \ds \hat{\bm E} = \big(\hat{\bm E}_{i,j}\big)_{i,j=1}^2 &  \hspace{-0.1in} \ds = :
   \begin{bmatrix}
\big( \p_{ v_1}  D_{1,1}, \p_{ v_1}  D_{1,2}\big)  &\big( \p_{ v_1}  D_{2,1}, \p_{v_1}  D_{2,2}\big)   \\[0.1in]
\big( \p_{v_2}  D_{1,1}, \p_{ v_2}  D_{1,2}\big)    & \big( \p_{ v_2}  D_{2,1}, \p_{ v_2}  D_{2,2}\big)
\end{bmatrix}.
%\\[0.25in]
%& \hspace{-0.1in} \; \ds \ds =\Big[\p_{\bm v} \bm D_{1,1}^\top,\p_{\bm v} \bm D_{1,2}^\top,\p_{\bm v} \bm D_{2,1}^\top,\p_{\bm v} \bm D_{2,2}^\top\Big]
\end{array}
\end{equation}

\begin{lemma}\label{Lem2} For some scalar function $\psi_c$,  $\bm E =\big(\p_{\bm v}D_{i,j}\big)_{i,j=1}^2$ and $\hat {\bm E}$ defined in \eqref{hatE} satisfy
\begin{equation}\begin{array}{c}\label{Lem2:eq1}
  \ds  \big[K(\bm x) (\bm E (\bm x, \bm v) \otimes  \nabla p )\nabla c\big] \cdot \nabla \psi_c = \nabla  p \cdot \big[K(\bm x) (\hat{\bm E} (\bm x, \bm v)\otimes  \nabla c) \nabla \psi_c \big]
  \end{array}
\end{equation}
on $\Omega \times [0,T]$.
\end{lemma}

\begin{proof}
By \eqref{Def:kron} and the definition of $\bm E$, we have
\begin{equation}\begin{array}{l}\label{Lem2:E}
\ds  \big[(\bm E (\bm x, \bm v) \otimes \nabla  p) \nabla c\big] \cdot \nabla \psi_c\ds =
\begin{bmatrix}
\p_{\bm v} D_{1,1}  \nabla p  &\p_{\bm v}  D_{1,2}  \nabla p  \\[0.1in]
\p_{\bm v} D_{2,1}   \nabla p  & \p_{\bm v}  D_{2,2} \nabla p
\end{bmatrix}\nabla c \cdot \nabla \psi_c
\\[0.25in]
\ds  = \Big[(\p_{v_1}D_{1,1}\p_x p + \p_{v_2}D_{1,1}\p_y p)\p_x c + (\p_{v_1} D_{1,2}\p_x p + \p_{v_2} D_{1,2}\p_y p)\p_y c  \Big] \p_x \psi_c\\[0.15in]
\ds \;\; +\Big[(\p_{v_1}D_{2,1}\p_x p + \p_{v_2}D_{2,1}\p_y p)\p_x c + (\p_{v_1} D_{2,2}\p_x p + \p_{v_2} D_{2,2}\p_y p)\p_y c  \Big] \p_y \psi_c,
\end{array}
\end{equation}
which, after proper rearrangement, becomes
\begin{equation}\begin{array}{l}\label{Lem2:E:e2}
\hspace{-0.15in}\ds  \big[(\bm E (\bm x, \bm v) \otimes \nabla  p) \nabla c\big] \cdot \nabla \psi_c \\[0.15in]
\hspace{-0.1in}\ds = \p_x p \Big[(\p_{v_1}D_{1,1}\p_x c + \p_{v_1} D_{1,2}\p_y c)\p_x \psi_c + (\p_{v_1}D_{2,1}\p_x c + \p_{v_1}D_{2,2}\p_y c)\p_y \psi_c  \Big] \\[0.15in]
\hspace{-0.1in}\ds \quad + \p_y p \Big[(\p_{v_2}D_{1,1}\p_x c + \p_{v_2}D_{1, 2}\p_y c)\p_x \psi_c + (\p_{v_2}D_{2,1}\p_x c + \p_{v_2}D_{2,2}\p_y c)\p_y \psi_c  \Big] \\[0.15in]
\hspace{-0.1in} \ds = \nabla p \cdot
\begin{bmatrix}
\big( \p_{ v_1} D_{1,1}, \p_{ v_1} D_{1, 2}\big)  \nabla c  &\big( \p_{ v_1} D_{2,1}, \p_{v_1} D_{2,2}\big)   \nabla c  \\[0.1in]
\big( \p_{v_2} D_{1,1}, \p_{ v_2} D_{1, 2}\big)   \nabla c  & \big( \p_{ v_2} D_{2,1}, \p_{ v_2} D_{2,2}\big) \nabla c
\end{bmatrix}\nabla \psi_c
 \\[0.2in]
\hspace{-0.1in} \ds = \nabla  p \cdot \big[ (\hat{\bm E} \otimes \nabla c) \nabla \psi_c \big]
\end{array}
\end{equation}
by  \eqref{Def:kron} and \eqref{hatE}. We multiply $K$ on both sides of \eqref{Lem2:E:e2} to complete the proof.
\end{proof}

\begin{lemma}\label{Lem3}  For some scalar function $\psi_c$ and under the boundary conditions
\begin{equation}\label{Lem3:bd}
 \bm v(\bm x, t) \cdot \bm n(\bm x)= (\bm D \nabla c) \cdot \bm n(\bm x)= (\bm D \nabla \psi_c) \cdot \bm n(\bm x) =0 ~~ \mbox{on}~~  \p \om \times [0,T],
\end{equation}
we have $\ds \big[ (\hat{\bm E} (\bm x, \bm v) \otimes \nabla c ) \nabla \psi_c\big] \cdot \bm n(\bm x) =0$  {on} $ \p \om \times [0,T]$ with $\hat {\bm E}(\bm x, \bm v)$ defined  in \eqref{hatE}.
\end{lemma}
\begin{proof}
Recall the definition for $\hat {\bm E}$ in \eqref{hatE}, i.e.,
\begin{equation}\begin{array}{l}\label{Lem3:hatE}
\hat {\bm E} =
\begin{bmatrix}
\big( \p_{ v_1} D_{1,1}, \p_{ v_1} D_{1, 2}\big)  &\big( \p_{ v_1} D_{2,1}, \p_{v_1} D_{2,2}\big)   \\[0.1in]
\big( \p_{v_2} D_{1,1}, \p_{ v_2} D_{1, 2}\big)    & \big( \p_{ v_2} D_{2,1}, \p_{ v_2} D_{2,2}\big)
\end{bmatrix},
\end{array}
\end{equation}
we then combine this with the split in \eqref{Lem:hatE} to accordingly split $  \hat {\bm E} := \hat {\bm  E^t}+ \hat  {\bm E^l}$  as follows:
\begin{equation}\begin{array}{l}\label{Lem3:Matrix2}
\hspace{-0.15in}\ds
{\footnotesize\hat{\bm E^t}: =\begin{bmatrix}
  \big([\bm E_d^t]_1,0\big) & \big(0, [\bm E_d^t]_1\big)  \\[0.1in]
  \big([\bm E_d^t]_2,0\big)  &  \big(0, [\bm E_d^t]_2\big)
\end{bmatrix},
\; \hat{\bm E^l}: =
\begin{bmatrix}
  \big([\bm E_{1,1}^l]_1,\p_{v_1} D_{1,2} \big) & \big( \p_{ v_1} D_{2,1}, [\bm E_{2,2}^l]_1\big)  \\[0.1in]
  \big([\bm E_{1,1}^l]_2,\p_{v_2} D_{1,2} \big) & \big( \p_{ v_2} D_{2,1}, [\bm E_{2,2}^l]_2\big)
\end{bmatrix}
}
\end{array},
\end{equation}
where $[\bm a]_i$ for $i =1, 2$ refers to the $i$-th element of the vector  $\bm a \in \mathbb R^{2}$.

We first show that $\big[(\hat{\bm E^t} \otimes \nabla c) \nabla \psi_c\big] \cdot \bm n(\bm x) =0$ on  $\p \om \times [0,T]$.
%By Lemma \ref{Lem:D} and $(\bm D \nabla c) \cdot \bm n(\bm x) =0 $ on  $\p \om \times [0,T]$ in \eqref{Model}, we have $\nabla c \cdot \bm n(\bm x) =0$ on  $\p \om \times [0,T]$.
By \eqref{bd:v}, \eqref{Def:kron}, \eqref{Lem:hatE} and \eqref{Lem3:Matrix2}, we have
\begin{equation}\begin{array}{cl}\label{Lem3:Et}
\ds  \big[(\hat{\bm E^t} \otimes \nabla c) \nabla \psi_c\big] \cdot \bm n(\bm x) & \hspace{-0.1in}\ds =\f{1}{|\bm v|}
\begin{bmatrix}
v_1 \p_x c   & v_1 \p_y c \\
v_2 \p_x c &v_2 \p_y c
\end{bmatrix}
\nabla \psi_c \cdot \bm n(\bm x)\\[0.175in]
& \hspace{-0.1in}\ds = \f{1}{|\bm v|} \big[v_1 \cos\theta \big(\p_x c \p_x \psi_c + \p_y c \p_y \psi_c \big) \\[0.1in]
& \ds \qquad \quad  +v_2 \sin\theta \big(\p_x c \p_x \psi_c + \p_y c \p_y \psi_c \big) \big]  = 0.
\end{array}
\end{equation}

We then evaluate $\big[(\hat {\bm E^l}\otimes \nabla c) \nabla \psi_c\big] \cdot \bm n(\bm x)$  on  $\p \om \times [0,T]$ by following a similar procedure in \eqref{Lem:El} as follows
\begin{equation}\begin{array}{cl}\label{Lem3:El}
&\hspace{-0.3in}  \ds\big[(\hat {\bm E^l}\otimes \nabla c) \nabla \psi_c\big] \cdot \bm n(\bm x)  \\[0.1in]
%\ds & \hspace{-0.1in}\ds =
%\begin{bmatrix}
%  \big([\bm E_{1,1}^l]_1,\p_{v_1} D_{1,2} \big) \cdot \nabla c & \big( \p_{ v_1} D_{2,1}, [\bm E_{2,2}^l]_1\big)\cdot \nabla c  \\[0.1in]
%  \big([\bm E_{1,1}^l]_2,\p_{v_2} D_{1,2} \big) \cdot \nabla c & \big( \p_{ v_2} D_{2,1}, [\bm E_{2,2}^l]_2\big) \cdot \nabla c
%\end{bmatrix}
% \nabla \psi_c \cdot \bm n(\bm x)\\[0.25in]
& \hspace{-0.1in}\ds
 =
 \begin{bmatrix}
 \big([\bm E_{1,1}^l]_1,\p_{v_1} D_{1,2} \big)  \nabla c  \p_x \psi_c  +  \big( \p_{ v_1} D_{2,1}, [\bm E_{2,2}^l]_1\big) \nabla c  \p_y \psi_c  \\[0.1in]
 \big([\bm E_{1,1}^l]_2,\p_{v_2} D_{1,2} \big) \nabla c\p_x \psi_c  +\big( \p_{ v_2} D_{2,1}, [\bm E_{2,2}^l]_2\big)  \nabla c\p_y \psi_c
\end{bmatrix} \cdot \bm n(\bm x)\\[0.25in]
& \hspace{-0.1in}\ds = \p_x \psi_c \Big[ \cos\theta \big([\bm E_{1,1}^l]_1,\p_{v_1} D_{1,2} \big)  \nabla c  + \sin\theta \big([\bm E_{1,1}^l]_2,\p_{v_2} D_{1,2} \big)  \nabla c\big)\Big] \\[0.1in]
& \hspace{-0.1in}\ds \qquad + \p_y \psi_c \Big[ \cos\theta\big( \big( \p_{ v_1} D_{2,1}, [\bm E_{2,2}^l]_1\big) \nabla c + \sin\theta  \big( \p_{ v_2} D_{2,1}, [\bm E_{2,2}^l]_2\big) \nabla c\Big] \\[0.1in]
& \hspace{-0.1in}\ds = :  \f{d_l-d_t}{|\bm v|^3} \big(\hat Q_1 \p_x \psi_c  + \hat Q_2 \p_y \psi_c\big).
\end{array}
\end{equation}
Here $\hat Q_1$ could be evaluated by \eqref{Lem:hatE}, \eqref{hatE},  and $\p_x c \cos\theta =- \p_y c \sin\theta $ on  $\p \om \times [0,T]$ derived from $(\bm D \nabla c) \cdot \bm n(\bm x) =0 $ on  $\p \om \times [0,T]$  by Lemma \ref{Lem:D}
\begin{equation}\begin{array}{cl}\label{Lem:hatQ1}
\hspace{-0.15in}\ds \hat Q_1
& \hspace{-0.125in} \ds = \cos\theta  \big[ v_1(v_1^2 + 2v_2^2) \p_x c + v_2^3 \p_y c \big]  + \sin\theta\big[-v_1^2v_2\p_x c + v_1^3 \p_y c \big] \\[0.1in]
\hspace{-0.15in}& \hspace{-0.125in} \ds = \big(-\sin\theta\p_y c\big) v_1(v_1^2 + 2v_2^2)   + \cos\theta  v_2^3 \p_y c   + \sin\theta\big[-v_1^2v_2\p_x c + v_1^3 \p_y c \big] \\[0.1in]
\hspace{-0.15in}& \hspace{-0.125in} \ds = - 2\sin\theta v_1v_2^2 \p_y c  + \cos\theta  v_2^3 \p_y c  - \sin\theta v_1^2v_2\p_x c,
\end{array}
\end{equation}
and $\hat Q_2$ could be evaluated in   a similar manner as
\begin{equation}\begin{array}{cl}\label{Lem:hatQ2}
\hspace{-0.15in}\ds \hat Q_2
& \hspace{-0.125in} \ds = \cos\theta  \big[ v_2^3 \p_x c - v_1 v_2^2 \p_y c \big]  + \sin\theta\big[v_1^3 \p_x c + v_2 (2v_1^2 +v_2^2) \p_y c \big] \\[0.1in]
& \hspace{-0.125in} \ds = \big(-\sin\theta \p_y c\big)  v_2^3 - \cos\theta v_1 v_2^2 \p_y c   + \sin\theta\big[v_1^3 \p_x c + v_2 (2v_1^2 +v_2^2) \p_y c \big] \\[0.1in]
\hspace{-0.15in}& \hspace{-0.125in} \ds = 2\sin\theta v_1^2v_2 \p_y c -  \cos\theta  v_1 v_2^2 \p_y c +\sin\theta v_1^3\p_x c.
\end{array}
\end{equation}
We then combine \eqref{Lem:hatQ1}--\eqref{Lem:hatQ2} to further reformulate \eqref{Lem3:El} as
\begin{equation}\begin{array}{l}\label{Lem3:Ele2}
  \ds \big[(\hat{\bm E^l}\otimes \nabla c) \nabla \psi_c\big] \cdot \bm n(\bm x)  \\[0.15in]
\ds \quad = \f{d_l-d_t}{|\bm v|^3} \Big[\p_x \psi_c \big(- 2\sin\theta v_1v_2^2 \p_y c  + \cos\theta  v_2^3 \p_y c  - \sin\theta v_1^2v_2\p_x c\big)  \\[0.15in]
\ds   \qquad \qquad  \qquad +\p_y \psi_c \big(2\sin\theta v_1^2v_2 \p_y c -  \cos\theta  v_1 v_2^2 \p_y c +\sin\theta v_1^3\p_x c\big)\Big]\\[0.125in]
\ds \quad = \f{d_l-d_t}{|\bm v|^3} \Big[ 2v_1 v_2 \p_y c \big(-v_2 \p_x \psi_c \sin\theta + v_1 \p_y \psi_c \sin\theta \big)\\[0.125in]
\ds \qquad \qquad  \qquad + v_2^2 \p_y c  \big( v_2 \p_x \psi_c \cos\theta - v_1\p_y \psi_c \cos\theta \big)\\[0.125in]
\ds \qquad \qquad  \qquad  + v_1^2 \p_x c \big(-v_2 \p_x \psi_c \sin\theta + v_1 \p_y \psi_c\sin\theta  \big)\Big]=0
\end{array}
\end{equation}
on  $\p \om \times [0,T]$. Here we have used  the fact that
\begin{equation}\begin{array}{l}\label{bdQ}
  \ds v_2 \p_x \psi_c \sin\theta =  \p_x \psi_c \big(v_2 \sin\theta\big)= - \p_x \psi_c(v_1  \cos\theta) = v_1(\p_y \psi_c \sin\theta),\\[0.1in]
  \ds v_2 \big( \p_x \psi_c \cos\theta\big) = - \p_y \psi_c\big(v_2  \sin\theta\big) = \p_y \psi_c \big(v_1  \cos\theta\big),
\end{array}
\end{equation}
which are derived from \eqref{bd:v}  and $\nabla \psi_c \cdot \bm n(\bm x) =0$ on  $\p \om \times [0,T]$ by Lemma \ref{Lem:D}. We thus complete the proof.
\end{proof}
\section{Application to optimal control}\label{Sect:Cond}

 We consider the optimal control problem for the incompressible miscible flow model \eqref{VtFDEs} for $d=2$ and then derive its  first-order optimality condition based on the boundary conditions in \S \ref{Sect:Lem}.

 We follow the conventional formulation  of the optimal control \cite{Gun} to modify the source terms in the first and third equations of  \eqref{VtFDEs}  as $q+h$, where  $h$ denotes the control variable that adjusts the injection and extraction rate of the injecting fluid from the admissible set
\begin{equation}\label{Uad}
U_{ad} :=  \big\{ h\in L^2(0,T;L^2(\Omega)): \int_\Omega h(\bm x,t) d\bm x =0 ~\text{ for } a. e. ~ t \in[0,T] \big\}.\end{equation}
 To facilitate the analysis, we substitute $ \bm v = -K(\bm x) \nabla p$ in $\nabla \cdot (c \bm v)$ in the first equation in the resulting system and  combine the second and the third equations   to reformulate
\begin{equation}\label{Model}\begin{array}{cl}
\hspace{-0.15in}	\ds \mathcal L_c(c, p, \bm v):= \phi \p_t c  \! -\!  \nabla \cdot (\bm D(\bm x, \bm v) \nabla c) \! -\! \nabla \cdot (c K(\bm x) \nabla p)= q+h,\,&\hspace{-0.075in} \mathrm{in} \,\Omega \times (0,T]; \\ [0.05in]
\ds   \ds \mathcal L_p( p) :=- \nabla \cdot (K(\bm x) \nabla p)  = q +h, ~&\hspace{-0.075in} \mathrm{in} ~\Omega \times [0,T]; \\[0.05in]
	c(\bm x,0) = 0, \quad &\hspace{-0.075in} \mathrm{on}~\om; \\[0.05in]
	 \nabla p(\bm x, t) \cdot \bm n(\bm x)  = (\bm D\nabla c)(\bm x,t) \cdot \bm n (\bm x) = 0,  & \hspace{-0.075in}\mathrm{on} \, \p \om \times [0,T].
\end{array}\end{equation}
We note that the combination of the  equation $\nabla \cdot \bm v = q +h$   and the boundary condition $\bm v(\bm x, t) \cdot \bm n(\bm x) = 0 ~\mathrm{on} ~ \p \om$ in \eqref{VtFDEs}  lead to the compatibility condition $\int_\Omega (q + h)(\bm x, t) d \bm x =0$ for a.e. $t \in [0, T]$,  which, combined with {\it assumption (b)}, further gives $\int_\Omega  h(\bm x, t) d \bm x =0$ for a.e. $t \in [0, T]$  as imposed in the admissible set \eqref{Uad}.

 We consider the optimal control problem \cite{Gun}
\begin{equation}\label{ObjFun}
\hspace{-0.15in}\ds \min_{h \in U_{ad}} J=\frac{1}{2}\|c-c_d\|^2_{L^2(0,T;L^2(\Omega))} + \frac{\alpha}{2}\|p-p_d\|^2_{L^2(0,T;L^2(\Omega))}  +\frac{\gamma}{2}\|h\|^2_{L^2(0,T;L^2(\Omega))},
\end{equation}
where  $c_d$ and $p_d$ denote the prescribed target states of $c$ and $p$, respectively, $\alpha \ge 0$ and $\gamma>0$ are parameters.
 The control mechanism is to manipulate $h$ to adjust the pressure of the fluid mixture and the concentration to the ideal distributions with the relatively low cost balanced by the parameter $\gamma$.

\begin{theorem}\label{thm:OptCond}
%Suppose assumptions (a)--(c) hold.
There exist adjoint states $\psi_c$ and $\psi_p$ with $ \int_\Omega \psi_p(\bm x, t) d \bm x =0$ such that $(c,p, \psi_c, \psi_p)$ satisfies the state equations in (\ref{Model}) and the adjoint state equations
\begin{equation}\label{AdjEq}\begin{array}{ll}
\hspace{-0.175in} \ds  \mathcal L^*_{c}(\psi_c, p, \bm v)  = c(\bm x,t;h)-c_d(\bm x,t),    & \ds \hspace{-0.1in}~ \mathrm{in} ~ \Omega\times[0,T); \\ [0.05in]
\hspace{-0.175in} \mathcal L^*_{p}(\psi_p, c, \psi_c, \bm v)  = \alpha\big(p(\bm x,t;h)-p_d(\bm x,t)\big),  & \ds \hspace{-0.1in}~ \mathrm{in} ~ \Omega\times[0,T]; \\ [0.05in]
\hspace{-0.175in}\ds \psi_c(\bm x,T)=0 ~ \mathrm{on} ~  \om, ~(\bm D\nabla \psi_c )\cdot \bm n (\bm x) =\nabla \psi_p\cdot \bm n (\bm x)= 0, & \ds \hspace{-0.1in} ~ \mathrm{on} ~ \p \om \times [0,T]
\end{array}\end{equation}
with
\begin{equation}\label{Adj:opr}\begin{array}{ll}
\hspace{-0.175in} \ds  \mathcal L^*_{c}(\psi_c, p, \bm v)  : = - \phi\p_t  \psi_c  - \nabla \cdot (\bm D(\bm x, \bm v) \nabla \psi_c)   +  K(\bm x) \nabla p \cdot \nabla \psi_c ,\\[0.1in]
\hspace{-0.175in} \ds \ds \mathcal L^*_{p}(\psi_p, c, \psi_c, \bm v) : =  \nabla \cdot \big( K(\bm x) (\hat {\bm E} (\bm x, \bm v)\otimes  \nabla c)  \nabla \psi_c \big) - \nabla \cdot (K(\bm x) \nabla \psi_p)  \\[0.1in]
\ds \qquad \qquad \qquad \qquad \qquad - \nabla \cdot (c  K(\bm x) \nabla \psi_c)
\end{array}\end{equation}
and
\begin{equation}\label{Vareq}
%\ds \ds h = -\f{\bar c \psi_c + \psi_p}{\gamma} +\f1{\gamma|\Omega|} \int_\Omega \bar c \psi_c\, d \bm x.
\ds \int_0^T \int_\Omega  \big( \psi_c + \psi_p+ \gamma h \big)(v-h)  d\bm x dt    =  0, \quad v \in U_{ad}.
\end{equation}

\end{theorem}
\begin{proof}
 For any $v \in U_{ad}$ and $0 < \varepsilon \ll 1$, let $\delta h := v-h$. We have $h + \varepsilon \delta h \in U_{ad}$ by the convexity of $U_{ad}$. Let $\delta_\varepsilon c := \big ( c(h+\varepsilon \delta h)-c(h) \big)/\varepsilon$ and $\delta_\varepsilon p := \big ( p(h+\varepsilon \delta h)-p(h) \big)/\varepsilon$, we learn from the first governing equation in \eqref{Model} that $\delta_\varepsilon \mathcal L_c = \delta h$ with
 \begin{equation}\label{Opr:equ:del}
 \begin{array}{l}
 \ds \delta_\varepsilon \mathcal L_c : =   \f1{\varepsilon}\big[ \mathcal L_c(c(h+\varepsilon \delta h), p(h+\varepsilon \delta h), \bm v(h+\varepsilon \delta h)) -  \mathcal L_c(c(h), p(h), \bm v(h)) \big].
\end{array}
\end{equation}
  To further simplify \eqref{Opr:equ:del}, we apply the mean-value theorem to get
\begin{equation}\label{Jv:rel1}\begin{array}{l}
 \ds \frac{1}{\varepsilon} \big[ \nabla \cdot \big(\bm D(\bm x, \bm v(h+\varepsilon \delta h)) \nabla c(h+\varepsilon \delta h) \big)
\ds -  \nabla \cdot (\bm D\big(\bm x, \bm v(h)) \nabla c(h) \big)\big]\\[0.075in]
  \ds   \qquad =   \f1{\varepsilon} \nabla \cdot \big ({\big(\bm D(\bm x, \bm v(h+\varepsilon \delta h))-\bm D(\bm x, \bm v(h))\big)} \nabla  c (h+\varepsilon \delta h)\big) \\[0.1in]
 \ds \qquad \qquad \quad + \nabla \cdot \big(\bm D(\bm x, \bm v(h)) \nabla \delta_\varepsilon c \big) \\[0.075in]
  \ds  \quad =  - \nabla \cdot \big (K(\bm x) (\bm E(\bm x, \bm \zeta) \otimes  \nabla \delta_\varepsilon p)  \nabla  c(h+\varepsilon \delta h) \big)+ \nabla \cdot \big(\bm D(\bm x, \bm v(h)) \nabla \delta_\varepsilon c \big)
\end{array}
\end{equation}
with the 2-by-2 block matrix $\bm E \in\mathbb R^{2\times 4} $  defined in Lemma \ref{Lem1}.
 We incorporate this with \eqref{Model}, \eqref{Opr:equ:del}, and
\begin{equation}\label{Jv:rel2}\begin{array}{l}
  \ds \frac{1}{\varepsilon} \big[ \nabla \cdot \big(c(h+\varepsilon \delta h) K(\bm x) \nabla p(h+\varepsilon \delta h)\big)
\ds -  \nabla \cdot \big(c(h) K(\bm x) \nabla p(h)\big)\big]\\[0.1in]
\ds \quad =\frac{1}{\varepsilon} \big[ \nabla \cdot \big(c(h+\varepsilon \delta h) K(\bm x) \nabla p(h+\varepsilon \delta h)\big) -
  \nabla \cdot \big(c(h+\varepsilon \delta h) K(\bm x) \nabla p(h)\big)\big]\\[0.1in]
  \ds \qquad + \frac{1}{\varepsilon} \big[ \nabla \cdot \big(c(h+\varepsilon \delta h) K(\bm x) \nabla p(h)\big) -
  \nabla \cdot \big(c(h) K(\bm x) \nabla p(h)\big)\big]\\[0.1in]
  \ds \quad  =    \nabla \cdot \big(c(h+\varepsilon \delta h) K(\bm x) \nabla \delta_\varepsilon p\big) +\nabla \cdot \big(\delta_ \varepsilon c K(\bm x) \nabla p(h)\big),
\end{array}
\end{equation}
 to deduce that  $\delta_\varepsilon c$  satisfies
\begin{equation}\label{Jv:e1}\begin{array}{cl}
 \mathcal L_c(\delta_\varepsilon c, p, \bm v) + \nabla \cdot \big (K(\bm x) \big(\bm E(\bm x, \bm \zeta) \otimes \nabla \delta_\varepsilon p\big)  \nabla  c(h+\varepsilon \delta h) \big)\\[0.1in]
 \ds\qquad \qquad \qquad - \nabla \cdot \big(c(h+\varepsilon \delta h) K(\bm x) \nabla \delta_\varepsilon p\big)= \delta h, & \hspace{-0.1in}~ \mathrm{in}~~ \Omega\times(0,T]; \\[0.05in]
%\hspace{-0.15in}\ds \phi(\p_t \delta_\varepsilon c + \kappa \p_t^{\alpha(t)} \delta_\varepsilon c)  -  \nabla \cdot (D(\bm x, \bm v) \nabla \delta_\varepsilon c) \\[0.05in]
%\hspace{-0.1in} + \nabla \cdot \Big (K(\bm x) \p_{\bm v } D(\bm x, \zeta) \cdot \nabla \delta_\varepsilon p \nabla  c(h+\varepsilon \delta h) \Big)\\[0.05in]
%\hspace{-0.05in} \ds  - \nabla \cdot (\delta_ \varepsilon c K(\bm x) \nabla p(h))  -   \nabla \cdot (c(h+\varepsilon \delta h) K(\bm x) \nabla \delta_\varepsilon p)= \bar c \delta h, & \hspace{-0.1in}~ \mathrm{in}~~ \Omega\times(0,T]; \\[0.05in]
\ds \ds  \delta_\varepsilon c(\bm x,0)=0, & \hspace{-0.1in}~  \mathrm{on} ~~  \om; \\[0.05in]
\ds \ds \nabla \delta_\varepsilon c(\bm x,t) \cdot \bm n (\bm x) = 0,  & \hspace{-0.1in} ~ \mathrm{on} ~~ \p \om \times [0,T].
\end{array}\end{equation}
We note that the boundary in \eqref{Jv:e1} is derived from $(\bm D(\bm x, \bm v) \nabla c)(\bm x,t) \cdot \bm n (\bm x) = 0$ on $ \p \om \times [0,T]$ in \eqref{Model} and thus, by Lemma \ref{Lem:D}, $ \nabla c(\bm x,t) \cdot \bm n (\bm x) = 0$ on $ \p \om \times [0,T]$.
In addition, we note from   \eqref{Model} that $\delta_\varepsilon \mathcal L_p = \mathcal L_p(\delta_\varepsilon p)$ and $\delta_\varepsilon p$ satisfies
\begin{equation}\label{Jv:e2}\begin{array}{cl}
 %\ds - \nabla \cdot (K(\bm x) \nabla \delta_\varepsilon p)  = \delta h, &~~ \mathrm{in} ~\Omega \times [0,T]; \\[0.05in]
 \ds \mathcal L_p(\delta_\varepsilon p) = -\nabla \cdot (K(\bm x) \nabla \delta_\varepsilon p) =\delta h,~~&\hspace{-0.1in}\ds  \mathrm{in} ~\Omega \times [0,T], \\[0.1in]
 \ds  \nabla \delta_\varepsilon p(\bm x, t) \cdot \bm n (\bm x) = 0,  ~~ &\hspace{-0.1in}\ds \mathrm{on} ~ \p \om \times [0,T].
\end{array}\end{equation}
We then use $c(h+\varepsilon \delta h) - c(h) = \varepsilon \delta_\varepsilon c$,  $p(h+\varepsilon \delta h) - p(h) = \varepsilon \delta_\varepsilon p$, \eqref{Uad}, and \eqref{AdjEq} to differentiate $\hat J(h) = J(c(h), p(h), h)$ to obtain
\begin{equation}\label{Jv:e4}\begin{array}{rl}
\hspace{-0.15in} \ds 0 & \hspace{-0.125in} \ds = \p_h \hat J(h) \delta h = \lim_{\varepsilon \rightarrow 0} \varepsilon^{-1}  \big ( \hat J(h+\varepsilon\delta h) -  \hat J(h) \big ) \\[0.1in]
& \hspace{-0.2in} \ds = \lim_{\varepsilon \rightarrow 0} (2\varepsilon)^{-1} \big [\|c(h+\varepsilon\delta h)-c_d\|^2_{L^2(L^2)}   +\alpha  \|p(h+\varepsilon\delta h)-p_d\|^2_{L^2(L^2)} \\[0.1in]
 & \hspace{-0.2in} \,+ \gamma \| h+\varepsilon\delta h\|^2_{L^2(L^2)} - \|c(h)-c_d\|^2_{L^2(L^2)} - \alpha \|p(h)-p_d\|^2_{L^2(L^2)} - \gamma \| h\|^2_{L^2(L^2)} \big ] \\[0.1in]
& \hspace{-0.2in} \ds = \int_0^T \int_\Omega (c(h)-c_d) \delta_\varepsilon c +  \alpha(p(h)-p_d) \delta_\varepsilon p \; d\bm x dt
+ \int_0^T\int_\Omega \gamma h \delta h \; d\bm x dt \\[0.1in]
& \hspace{-0.2in} \ds = \int_0^T \int_\Omega \mathcal L^*_{c}(\psi_c, p, \bm v) \delta_\varepsilon c + \mathcal L^*_{p}(\psi_p, c, \psi_c, \bm v)  \delta_\varepsilon p + \gamma h \delta h \,d\bm x dt.
%& \hspace{-0.2in} \ds = \int_0^T \int_\Omega  \big(\phi(-\p_t \psi_c \!+\! \kappa \, {}^R\hat\p_t^{\alpha(t)} \psi_c) \!-\!  \nabla \cdot (D(\bm x, \bm v) \nabla \psi_c)  \!+ \! K(\bm x) \nabla p \cdot \nabla \psi_c \big ) \delta_\varepsilon c \, d\bm x dt\\[0.15in]
%& \hspace{-0.2in} \ds  + \int_0^T \int_\Omega  \Big( \nabla \cdot (\p_{\bm v} D (\bm x, \bm v) K(\bm x) \nabla c \cdot \nabla \psi_c)-  \nabla \cdot (K(\bm x) \nabla \psi_p) \\[0.1in]
%& \qquad \qquad  \ds  -  \nabla \cdot (c  K(\bm x) \nabla \psi_c)  \Big ) \delta_\varepsilon p \, d\bm x dt  +   \int_0^T\int_\Omega \gamma h \delta h \, d\bm x dt.
\end{array}\end{equation}
We combine the boundary condition for $\delta_\varepsilon c$ in \eqref{Jv:e1} and Lemma \ref{Lem:D} to obtain $(\bm D(\bm x, \bm v) \nabla \delta_\varepsilon c ) (\bm x,t) \cdot \bm n (\bm x) = 0$ on $\p \om \times [0,T]$. We utilize this with \eqref{Model} and   \eqref{Adj:opr} to integrate the first term on the right-hand side of \eqref{Jv:e4} by parts to arrive at
\begin{equation}\label{Jv:e5a}\begin{array}{l}
 \ds \int_0^T \int_\Omega \mathcal L^*_{c}(\psi_c, p, \bm v) \delta_\varepsilon c \, d\bm x dt\\[0.15in]
%=\int_0^T \int_\Omega  \big(\phi(-\p_t \psi_c + \kappa \, {}^R\hat\p_t^{\alpha(t)} \psi_c) - \nabla \cdot (D(\bm x,  \bm v) \nabla \psi_c)  \!+ \! K(\bm x) \nabla p \cdot \nabla \psi_c \big ) \delta_\varepsilon c \, d\bm x dt\\[0.15in]
\ds  = \int_0^T \int_\Omega  \big( \phi\p_t \delta_\varepsilon c -  \nabla \cdot (\bm D(\bm x, \bm v) \nabla \delta_\varepsilon c )  -  \nabla \cdot (  \delta_\varepsilon c K(\bm x) \nabla p ) \big ) \psi_c \; d\bm x dt\\[0.15in]
 \ds = \int_0^T \int_\Omega \mathcal L_{c}(\delta_\varepsilon c, p, \bm v) \psi_c \, d\bm x dt.
\end{array}\end{equation}

Due to the coupling and nonlinearity of the forward problem \eqref{Model} as well as the complexity of the Bear-Scheidegger tensor,  its optimal control problem involves additional two terms in \eqref{Jv:e1}, i.e. the second and the third left-hand side terms of (\ref{Jv:e1}), that exhibit complicated structure. Consequently, several complex boundary conditions are required in the derivation of the first-order optimality condition, which is not typically encountered in the forward problem \eqref{Model} and optimal control of its linear analogues and thus requires substantial investigations.
Thus, novel lemmas are proved in \S \ref{Sect:Lem} to accommodate this issue and are applied in the following critical derivation
\begin{equation}\begin{array}{l}\label{Int:E}
\ds \int_0^T \int_\Omega  \nabla  \cdot \big(K(\bm x) ( \hat{\bm E} (\bm x, \bm v)  \otimes \nabla c) \nabla \psi_c\big) \delta_\varepsilon p \,d \bm x dt\\[0.125in]
\ds  = - \int_0^T \int_\Omega  \nabla \delta_\varepsilon p \cdot \big(K(\bm x) (\hat{\bm E} (\bm x, \bm v) \otimes \nabla c) \nabla \psi_c  \big) d \bm x dt\\[0.15in]
\ds= - \int_0^T \int_\Omega \big[K(\bm x) (\bm E (\bm x, \bm v) \otimes \nabla \delta_\varepsilon p) \nabla c \big]\cdot \nabla \psi_c  d \bm x dt\\[0.15in]
\ds =\int_0^T \int_\Omega \nabla \cdot \big(K(\bm x) (\bm E (\bm x, \bm v) \otimes \nabla \delta_\varepsilon p ) \nabla c \big)\psi_c d \bm x dt,
\end{array}
\end{equation}
where the first equality utilizes  $ \big[ (\hat{\bm E} (\bm x, \bm v) \otimes \nabla c ) \nabla \psi_c\big] \cdot \bm n(\bm x) =0$ on $\p \om \times [0,T]$ in  Lemma \ref{Lem3},   the second equality applies Lemma \ref{Lem2} with $p$ replaced by $\delta_\varepsilon p$, and the last equality utilizes the relation $\big[(\bm E (\bm x, \bm v) \otimes \nabla \delta_\varepsilon p) \nabla c \big] \cdot \bm n(\bm x) =0$  on $ \p \om \times [0,T]$ from   Lemma \ref{Lem1} since $\delta_\varepsilon p$ in \eqref{Jv:e2}   has the same boundary condition as $p$ in \eqref{Model}.

Lemma \ref{Lem:D} combined with $(\bm D \nabla \psi_c ) \cdot \bm n(\bm x)=0$ on $ \p \om \times [0,T]$ in \eqref{AdjEq} yields that $\nabla \psi_c \cdot \bm n(\bm x)=0$ on $ \p \om \times [0,T]$. We employ this with \eqref{Jv:e2} and \eqref{Int:E}   to evaluate the second term on the right-hand side of \eqref{Jv:e4}
\begin{equation}\label{Jv:e5b}\begin{array}{rl}
& \hspace{-0.175in} \ds   \int_0^T \int_\Omega \mathcal L^*_{p}(\psi_p, c, \psi_c, \bm v)  \delta_\varepsilon p \, d\bm x dt = \int_0^T \int_\Omega  \Big(\nabla \cdot \big(K(\bm x) (\hat{\bm E} (\bm x, \bm v)  \otimes \nabla c)  \nabla \psi_c\big) \\[0.1in]
& \hspace{-0.175in} \qquad  \qquad  \qquad \ds -  \nabla \cdot (K(\bm x) \nabla \psi_p)   -  \nabla \cdot (c  K(\bm x) \nabla \psi_c)  \Big ) \delta_\varepsilon p \,  d\bm x dt\\[0.1in]
& \hspace{-0.175in} \ds  = \int_0^T \int_\Omega  \Big(  \nabla \cdot \big(K(\bm x) (\bm E (\bm x, \bm v)\otimes  \nabla \delta_\varepsilon p) \nabla c \big )  -  \nabla \cdot \big(c  K(\bm x) \nabla \delta_\varepsilon p\big) \Big)\psi_c \,  d\bm x dt\\[0.1in]
& \hspace{-0.175in} \ds \qquad +\int_0^T \int_\Omega  -\nabla \cdot \big(K(\bm x) \nabla \delta_\varepsilon p\big) \psi_p \,  d\bm x dt\\[0.1in]
& \hspace{-0.15in} \ds = \int_0^T \int_\Omega   \Big(  \nabla \cdot \big(K(\bm x) (\bm E (\bm x, \bm v) \otimes \nabla \delta_\varepsilon p) \nabla c \big)  -  \nabla \cdot (c  K(\bm x) \nabla \delta_\varepsilon p) \Big) \psi_c \,  d\bm x dt\\[0.1in]
& \hspace{-0.15in} \ds \qquad + \int_0^T \int_\Omega \ds \mathcal L_p(\delta_\varepsilon p)\psi_p  \,  d\bm x dt.
\end{array}\end{equation}
We then incorporate the following equality resulting from \eqref{Jv:e1}
$$
\begin{array}{l}
 \mathcal L_c(\delta_\varepsilon c, p, \bm v) = - \nabla \cdot (K(\bm x)(\bm E(\bm x, \bm \zeta) \otimes \nabla \delta_\varepsilon p)  \nabla  c(h+\varepsilon \delta h) \big)\\[0.125in]
 \ds\qquad \qquad \qquad \qquad \qquad  + \nabla \cdot \big(c(h+\varepsilon \delta h) K(\bm x) \nabla \delta_\varepsilon p\big)+ \delta h
 \end{array}
$$
with  \eqref{Jv:e2}, \eqref{Jv:e5a}  as well as \eqref{Jv:e5b} to rewrite the first two terms on the right-hand side terms of \eqref{Jv:e4} as follows
\begin{equation*}\label{Jv:e5c}\begin{array}{ll}
\hspace{-0.2in} \ds \int_0^T \int_\Omega \Big( \mathcal L_{c}(\delta_\varepsilon c, p, \bm v) \!+\! \nabla \cdot \big(K(\bm x)  (\bm E (\bm x, \bm v) \otimes \delta_\varepsilon p ) \nabla c\big)  \!- \!\nabla \cdot (c  K(\bm x) \nabla \delta_\varepsilon p)  \Big) \psi_c \,d\bm x dt\\[0.1in]
\ds \qquad + \int_0^T \int_\Omega \mathcal L_p(\delta_\varepsilon p)\psi_p \,d\bm x dt
   =\int_0^T \int_\Omega \delta h( \psi_c +  \psi_p )\,d \bm x dt  \\[0.15in]
  \ds +\int_0^T \int_\Omega  \Big[\nabla \cdot \big(K(\bm x) ( \bm E (\bm x, \bm v) \otimes \delta_\varepsilon p)  \nabla c\big)     \\[0.15in]
   \ds \qquad \qquad \quad \ds -\nabla \cdot \big (K(\bm x) (\bm E(\bm x, \bm \zeta) \otimes\nabla \delta_\varepsilon p) \nabla  c(h+\varepsilon \delta h)\big)\Big] \psi_c\, d\bm x dt \\[0.1in]
 \ds +\int_0^T \int_\Omega  \Big[\nabla \cdot \big (c(h+\varepsilon \delta h) K(\bm x) \nabla \delta_\varepsilon p\big) - \nabla \cdot \big(c  K(\bm x) \nabla \delta_\varepsilon p\big) \Big] \psi_c \, d\bm x dt,
\end{array}\end{equation*}
 the right-hand side of which  converges to $\int_0^T \int_\Omega  \delta h \big( \psi_c +  \psi_p \big)\, d\bm x dt$  as $\varepsilon \rightarrow 0$. Thus we combine the above relation with \eqref{Jv:e4} to conclude that
\begin{equation}\label{Jv:e6}\begin{array}{rl}
 \ds \int_0^T \int_\Omega \delta h \big( \psi_c + \psi_p+ \gamma h \big)  d\bm x dt = \ds \int_0^T \int_\Omega  \big(\psi_c + \psi_p+ \gamma h \big)(v-h)  d\bm x dt =  0
\end{array}\end{equation}
for $ v \in U_{ad}$, which completes the proof.
\end{proof}

\section{Complex boundary conditions in 3D}\label{Sect:Lem3d}
We  prove complex boundary conditions for the three-dimensional case, which requires more technical derivations.

\begin{lemma}\label{Lem3d:D}
The equivalent relation \eqref{LemD:eq1} in Lemma \ref{Lem:D} is valid for $\bm D$ defined in  \eqref{D} with $d =3$.
\end{lemma}
\begin{proof}
Recall the  relation \eqref{LemD:eq1}, we first prove that $(\bm D \nabla g) \cdot \bm n(\bm x) =0 $ on $ \p \om \times [0,T]$  implies $\nabla g \cdot \bm n(\bm x) =0 $ on $ \p \om \times [0,T]$ for some scalar function $g$. By \eqref{D}, we have
\begin{equation}\begin{array}{cl}\label{LemD3d:Dg}
\hspace{-0.15in}\ds (\bm D \nabla g) \cdot \bm n(\bm x) = \big(\phi  + d_t |\bm v|\big) \nabla g \cdot \bm n(\bm x) +  \f{(d_l-d_t)}{|\bm v| } \big((v_i v_j)_{i,j=1}^3 \nabla g \big)\cdot \bm n(\bm x).
\end{array}
\end{equation}
By the spherical polar coordinates, the outward unit normal vector $\bm n$ could be expressed as $\bm n := (\cos\theta\sin\varphi, \sin\theta\sin\varphi, \cos\varphi  )$ for some $\theta$  and $\varphi$  depending on $\bm x \in \p \om$. We incorporate this to reformulate
\begin{equation}\begin{array}{cl}\label{LemD3d:Dg:e1}
\hspace{-0.15in}\ds \big((v_i v_j)_{i,j=1}^3 \nabla g \big)\cdot \bm n(\bm x) & \ds \hspace{-0.1in}=
\begin{bmatrix}
  v_1^2 \p_x g + v_1 v_2 \p_y g +  v_1 v_3 \p_z g  \\[0.1in]
  v_1 v_2 \p_x g + v_2^2 \p_y g +  v_2 v_3 \p_z g \\[0.1in]
  v_1 v_3 \p_x g + v_2 v_3 \p_y g +   v_3^2 \p_z g
\end{bmatrix} \cdot \bm n(\bm x)\\[0.4in]
 & \ds \hspace{-0.1in}=\p_x g v_1 \big(v_1 \cos\theta\sin\varphi + v_2 \sin\theta \sin\varphi + v_3 \cos \varphi\big) \\[0.05in]
  & \ds \hspace{-0.1in} \quad  + \p_y g v_2 \big(v_1 \cos\theta\sin\varphi + v_2 \sin\theta \sin\varphi + v_3 \cos \varphi\big)\\[0.05in]
  & \ds \hspace{-0.1in} \quad  + \p_z g v_3 \big(v_1 \cos\theta\sin\varphi + v_2 \sin\theta \sin\varphi + v_3 \cos \varphi\big)=0,
\end{array}
\end{equation}
where we have used the fact that
\begin{equation}\label{bd3d:v}\begin{array}{l}
v_1 \cos\theta\sin\varphi +   v_2 \sin\theta\sin\varphi + v_3 \cos \varphi=0 ~~\mbox{on}  ~~ \p \om \times [0,T]
\end{array}
\end{equation}
 derived from $\bm v \cdot \bm n(\bm x)=0$ on $ \p \om \times [0,T]$  in \eqref{VtFDEs}. By \eqref{LemD3d:Dg}--\eqref{LemD3d:Dg:e1}, $ (\bm D \nabla g) \cdot \bm n(\bm x) =0$ on $ \p \om \times [0,T]$  implies $\big(\phi  + d_t |\bm v|\big) \nabla g \cdot \bm n(\bm x)$ =0 on $ \p \om \times [0,T]$, which further gives  $\nabla g \cdot \bm n(\bm x) =0 $ on $ \p \om \times [0,T]$  since $\phi  + d_t |\bm v|  \ge \phi > \phi_* >0$ by {\it assumption (b)}.

Conversely, if $\nabla g \cdot \bm n(\bm x) =0 $ on $ \p \om \times [0,T]$, \eqref{LemD3d:Dg} combined with \eqref{LemD3d:Dg:e1} gives $(\bm D \nabla g) \cdot \bm n(\bm x)=0$ on $ \p \om \times [0,T]$, and we thus complete the proof.
\end{proof}

\begin{lemma}\label{Lem3d1}
Under the boundary conditions \eqref{Lem1:bd} in Lemma \ref{Lem1}, we arrive at  $\ds \big[ (\bm E (\bm x, \bm v) \otimes \nabla p) \nabla c\big] \cdot \bm n(\bm x) =0$  \mbox{on} $\p \om \times [0,T]$, where we similarly denote the  3-by-3 block matrix  by $\bm E :=\big(\p_{\bm v}D_{i,j}\big)_{i,j=1}^3$  for the sake of simplicity with its entries given by
\begin{equation}\begin{array}{cl}\label{DhatE3d}
 \hspace{-0.2in}\ds \p_{\bm v} D_{1,1} &\hspace{-0.1in}\ds  : = \p_{\bm v}\big( dt |\bm v| +  (d_l-d_t){v_1^2}/{|\bm v|}\big)=  \bm E_d + \bm E_{l,1}, \\[0.125in]
\hspace{-0.2in} \ds \p_{\bm v}  D_{1,2}  &\hspace{-0.1in}\ds : = (d_l-d_t) \p_{\bm v}\big(  {v_1 v_2}/{|\bm v|}\big) \!= \! \f{(d_l-d_t)}{{|\bm v|^3}} \big({v_2(v_2^2 + v_3^2)}, {v_1(v_1^2 + v_3^2)}, {-v_1v_2v_3}\big),\\[0.125in]
\hspace{-0.2in} \ds \p_{\bm v}  D_{1,3}  &\hspace{-0.1in}\ds : = (d_l-d_t) \p_{\bm v}\big(  {v_1 v_3}/{|\bm v|}\big) \!= \! \f{(d_l-d_t)}{{|\bm v|^3}}  \big({v_3(v_2^2 + v_3^2)}, {-v_1v_2 v_3}, {v_1(v_1^2 + v_2^2)}\big),\\[0.125in]
\hspace{-0.2in} \ds \p_{\bm v}  D_{2,2} &\hspace{-0.1in}\ds  : = \p_{\bm v}\big( dt |\bm v| +  (d_l-d_t){v_2^2}/{|\bm v|}\big) =   \bm E_d +\bm E_{l,2}, \\[0.125in]
\hspace{-0.2in} \ds \p_{\bm v}  D_{2,3}  &\hspace{-0.1in}\ds : = (d_l-d_t) \p_{\bm v}\big(  {v_2 v_3}/{|\bm v|}\big) \!= \! \f{(d_l-d_t)}{{|\bm v|^3}}  \big({-v_1v_2 v_3 }, {v_3(v_1^2 + v_3^2)}, {v_2(v_1^2 + v_2^2)}\big),\\[0.125in]
\hspace{-0.2in} \ds \p_{\bm v}  D_{3,3} &\hspace{-0.1in}\ds  : = \p_{\bm v}\big( dt |\bm v| +  (d_l-d_t){v_3^2}/{|\bm v|}\big) =   \bm E_d + \bm E_{l,3},\\[0.125in]
\hspace{-0.2in} \ds  \p_{\bm v} D_{2,1} & \hspace{-0.1in}\ds =  \p_{\bm v}  D_{1,2}, \quad \p_{\bm v} D_{3,1}=  \p_{\bm v}  D_{1,3}, \quad  \p_{\bm v} D_{3,2}=  \p_{\bm v}  D_{2,3},
\end{array}
\end{equation}
and $\otimes$ is accordingly defined  as   a Kronecker type product of a 3-by-3 block matrix and a column vector in $\mathbb R^{3\times 1}$. Here $\bm E_d$, $\bm E_{l,1}$, $\bm E_{l,2}$ and $\bm E_{l,3}$ are defined as follows
\begin{equation}\begin{array}{cl}\label{DhatE3d1}
\ds \bm E_d &\hspace{-0.1in}\ds  :=    \p_{\bm v} \big(d_t |\bm v| \big)  =   \f{dt}{|\bm v|} \big({v_1}, {v_2}, {v_3}\big), \\[0.1in]
\bm E_{l,1} &\hspace{-0.1in}\ds  := (d_l-d_t) \p_{\bm v}\big(  {v_1^2}/{|\bm v|}\big) = \f{(d_l-d_t)}{{|\bm v|^3}}  \big({v_1(v_1^2 + 2 v_2^2 + 2 v_3^2)}, {-v_1^2 v_2}, {-v_1^2 v_3}\big),\\[0.15in]
\bm E_{l,2} &\hspace{-0.1in}\ds  := (d_l-d_t) \p_{\bm v}\big(  {v_2^2}/{|\bm v|}\big) = \f{(d_l-d_t)}{{|\bm v|^3}}  \big({-v_1 v_2^2 }, {v_2(2 v_1^2 +  v_2^2 + 2 v_3^2)}, {- v_2^2 v_3 }\big),\\[0.15in]
\bm E_{l,3} &\hspace{-0.1in}\ds  := (d_l-d_t) \p_{\bm v}\big(  {v_3^2}/{|\bm v|}\big) = \f{(d_l-d_t)}{{|\bm v|^3}}  \big({-v_1 v_3^2 }, {- v_2 v_3^2 }, {v_3(2 v_1^2 +  2 v_2^2 +  v_3^2)}\big).
\end{array}
\end{equation}
\end{lemma}

\begin{proof}
We incorporate \eqref{DhatE3d} to split $\bm E =  \bm E^1 +  \bm E^2 $ as the sum of two 3-by-3  block  matrices with $\bm E^1: = \diag(\bm E_d, \bm E_d, \bm E_d)$ and
\begin{equation}\begin{array}{cl}\label{Lem3d:Matrix}
\bm E^2 = \big(\bm E^2_{i,j}\big)_{i,j=1}^3: =
\begin{bmatrix}
  \bm E_{l,1} & \p_{\bm v}  D_{1,2}   & \p_{\bm v}  D_{1,3} \\[0.05in]
  \p_{\bm v}  D_{2,1} & \bm E_{l,2} &   \p_{\bm v}  D_{2,3}\\[0.05in]
  \p_{\bm v}  D_{3,1} & \p_{\bm v}  D_{3,2} &  \bm E_{l,3}
\end{bmatrix}
\end{array}.
\end{equation}
We first show that $\big[ (\bm E^1 \otimes \nabla p) \nabla c\big ]\cdot \bm n(\bm x) =0$ on $\p \om \times [0,T]$. By Lemma \ref{Lem3d:D}, $(\bm D \nabla c) \cdot \bm n(\bm x) =0 $ on $ \p \om \times [0,T]$  implies that $\nabla c \cdot \bm n(\bm x) =0$ on $ \p \om \times [0,T]$. By \eqref{Def:kron} and \eqref{Lem3d:Matrix}, we have
\begin{equation}\begin{array}{cl}\label{Lem3d:Et}
\ds  \big[ (\bm E^1 \otimes  \nabla p) \nabla c\big ]\cdot \bm n(\bm x)\ds  & \ds \hspace{-0.1in}= \big[\diag(\bm E_d \nabla p, \bm E_d\nabla p, \bm E_d\nabla p)  \nabla c\big] \cdot \bm n(\bm x) \\[0.1in]
& \ds \hspace{-0.1in} \ds = \big(\bm E_d   \nabla p\big) \nabla c \cdot \bm n(\bm x) = 0
\end{array}
\end{equation}
on $ \p \om \times [0,T]$.
By \eqref{Lem3d:Matrix}, we  evaluate $\big[(\bm E^2 \otimes\nabla p) \nabla c\big] \cdot \bm n(\bm x) $ on $ \p \om \times [0,T]$ as follows
\begin{equation}\begin{array}{cl}\label{Lem3d:El}
 & \hspace{-0.2in}\ds \ds  \big[(\bm E^2 \otimes \nabla p) \nabla c\big] \cdot \bm n(\bm x)   =
\begin{bmatrix}
\bm E_{l,1}  \nabla p  &\p_{\bm v}  D_{1,2} \nabla p  & \p_{\bm v}  D_{1,3} \nabla p  \\[0.05in]
\p_{\bm v} D_{2,1} \nabla p  & \bm E_{l,2} \nabla p & \p_{\bm v}  D_{2,3} \nabla p \\[0.05in]
\p_{\bm v} D_{3,1} \nabla p  & \p_{\bm v} D_{3,2} \nabla p & \bm E_{l,3} \nabla p \\[0.05in]
\end{bmatrix}
\nabla c \cdot \bm n(\bm x)\\[0.3in]
& \hspace{-0.1in}\ds
=
 \begin{bmatrix}
\big(\bm E_{l,1}  \nabla p\big)\p_x c  + \big(\p_{\bm v}  D_{1,2}  \nabla p\big)\p_y c + \big(\p_{\bm v}  D_{1,3} \nabla p\big) \p_z c  \\[0.1in]
\big(\p_{\bm v}  D_{2,1}   \nabla p\big)\p_x c  + \big( \bm E_{l,2}  \nabla p\big)\p_y c + \big(\p_{\bm v}  D_{2,3} \nabla p \big) \p_z c\\[0.1in]
\big(\p_{\bm v} D_{3,1} \nabla p\big)\p_x c  + \big(\p_{\bm v} D_{3,2} \nabla p\big) \p_y c + \big(\bm E_{l,3} \nabla p\big) \p_z c
\end{bmatrix} \cdot \bm n(\bm x),
\end{array}
\end{equation}
which could be reformulated by
\begin{equation}\begin{array}{l}\label{Lem3d:El2}
 \hspace{-0.15in}\ds \ds  \big[(\bm E^2 \otimes \nabla p) \nabla c\big] \cdot \bm n(\bm x) \\[0.1in]
\ds = \p_x c \Big[ \cos\theta \sin\varphi\big(\bm E_{l,1} \nabla p\big) + \sin\theta\sin\varphi  \big(\p_{\bm v}  D_{2,1}   \nabla p\big) + \cos \varphi \big(\p_{\bm v} D_{3,1} \nabla p\big) \Big] \\[0.1in]
\ds  \quad + \p_y c \Big[ \cos\theta \sin\varphi \big(\p_{\bm v}  D_{1,2}  \nabla p\big) + \sin\theta\sin\varphi  \big(\bm E_{l,2}  \nabla p\big)+ \cos \varphi \big(\p_{\bm v} D_{3,2} \nabla p\big)\Big] \\[0.1in]
 \ds  \quad + \p_z c \Big[ \cos\theta \sin\varphi \big(\p_{\bm v}  D_{1,3} \nabla p\big) + \sin\theta\sin\varphi  \big(\p_{\bm v}  D_{2,3} \nabla p \big)  + \cos \varphi \big(\bm E_{l,3} \nabla p\big)\Big] \\[0.1in]
\ds = : \f{d_l-d_t}{|\bm v|^3}\big( H_1 \p_x c  + H_2 \p_y c +  H_3 \p_z c\big).
\end{array}
\end{equation}
By  \eqref{DhatE3d}--\eqref{DhatE3d1} and \eqref{Lem3d:El2}, we have
$$
\begin{array}{cl}
\hspace{-0.15in}\ds H_1
& \hspace{-0.125in} \ds = (v_1 \cos\theta \sin \varphi ) \big[ (v_1^2 + 2v_2^2 + 2v_3^2 ) \p_x p -v_1 v_2 \p_y p -v_1 v_3 \p_z p \big] \\[0.1in]
\hspace{-0.15in}& \hspace{-0.125in} \ds  + \sin\theta\sin \varphi  \big[v_2 (v_2^2 + v_3^2)\p_x p + v_1(v_1^2 + v_3^2) \p_y p - v_1 v_2 v_3 \p_z p \big]\\[0.1in]
\hspace{-0.15in}& \hspace{-0.125in} \ds  + \cos \varphi  \big[v_3 (v_2^2 + v_3^2)\p_x p - v_1 v_2 v_3 \p_y p + v_1 (v_1 ^2 + v_2^2 )\p_z p \big],
\end{array}
$$
which   could be further simplified  as follows
\begin{equation}\begin{array}{cl}\label{Lem3d:Q1}
\hspace{-0.15in}\ds H_1
& \hspace{-0.125in} \ds \overset{\mbox{(i)}}= -(v_2 \sin\theta \sin \varphi ) \big[ (v_1^2 + 2v_2^2 + 2v_3^2 ) \p_x p -v_1 v_2 \p_y p -v_1 v_3 \p_z p \big] \\[0.1in]
& \hspace{-0.125in} \qquad \ds  -(v_3 \cos \varphi ) \big[ (v_1^2 + 2v_2^2 + 2v_3^2 ) \p_x p -v_1 v_2 \p_y p -v_1 v_3 \p_z p \big] \\[0.1in]
\hspace{-0.15in}& \hspace{-0.125in} \qquad \ds \ds  + \sin\theta\sin \varphi  \big[v_2 (v_2^2 + v_3^2)\p_x p + v_1(v_1^2 + v_3^2) \p_y p - v_1 v_2 v_3 \p_z p \big]\\[0.1in]
\hspace{-0.15in}& \hspace{-0.125in}\qquad \ds  \ds  + \cos \varphi  \big[v_3 (v_2^2 + v_3^2)\p_x p - v_1 v_2 v_3 \p_y p + v_1 (v_1 ^2 + v_2^2 )\p_z p \big]\\[0.1in]
& \hspace{-0.125in} \ds =  \sin\theta \sin \varphi \big[ -v_2 |\bm v|^2 \p_x p + v_1 |\bm v|^2 \p_y p \big] + \cos\varphi\big[-v_3 |\bm v|^2 \p_x p + v_1 |\bm v|^2 \p_z p \big]\\[0.01in]
& \hspace{-0.125in} \ds  \overset{\mbox{(ii)}} =  \big(v_1 \cos\theta \sin \varphi + v_3 \cos \varphi\big)  |\bm v|^2 \p_x p + \sin\theta \sin \varphi  v_1 |\bm v|^2 \p_y p   \\[0.1in]
& \hspace{-0.125in} \ds \qquad  + \cos\varphi\big[-v_3 |\bm v|^2 \p_x p + v_1 |\bm v|^2 \p_z p \big] \\[0.1in]
& \hspace{-0.125in} \ds = v_1 |\bm v|^2 \big[\p_x p \cos\theta\sin\varphi +   \p_y p \sin\theta\sin\varphi + \p_z p \cos \varphi\big] =0,
\end{array}
\end{equation}
where   (i) and (ii) are by the boundary   condition \eqref{bd3d:v} and the last equality employs $\nabla p \cdot \bm n(\bm x) =0$ on $ \p \om \times [0,T]$ in \eqref{Model}.
Similarly, we evaluate $H_2$ in \eqref{Lem3d:El2} by  \eqref{DhatE3d}--\eqref{DhatE3d1}, and   utilize the boundary   condition \eqref{bd3d:v} to derive  (iii) and (iv) to find
\begin{equation}\begin{array}{cl}\label{Lem3d:Q2}
\hspace{-0.15in} \ds H_2 & \hspace{-0.125in} \ds =  \cos\theta \sin \varphi  \big[ v_2(v_2^2 + v_3^2 ) \p_x p + v_1 (v_1^2 + v_3^2 ) \p_y p -v_1 v_2 v_3 \p_z p \big] \\[0.1in]
\hspace{-0.15in}& \hspace{-0.125in} \ds \qquad \ds  + ( v_2 \sin\theta\sin \varphi)  \big[- v_1 v_2 \p_x p + (2 v_1^2 + v_2^2 +  2 v_3^2) \p_y p - v_2 v_3 \p_z p \big]\\[0.1in]
\hspace{-0.15in}& \hspace{-0.125in} \ds \qquad \ds  + \cos \varphi  \big[- v_1 v_2 v_3  \p_x p + v_3 (v_1^2 + v_3^2) \p_y p + v_2 (v_1 ^2 + v_2^2 )\p_z p \big]\\[0.01in]
\hspace{-0.15in}& \hspace{-0.125in} \ds  \overset{\mbox{(iii)}}= \cos\theta \sin \varphi \big[ v_2 |\bm v|^2 \p_x p - v_1 |\bm v|^2 \p_y p\big] + \cos\varphi\big[-v_3 |\bm v|^2 \p_y p + v_2 |\bm v|^2 \p_z p \big] \\[0.01in]
\hspace{-0.15in}& \hspace{-0.125in} \ds  \overset{\mbox{(iv)}}= v_2 |\bm v|^2 \big[\p_x p \cos\theta\sin\varphi +   \p_y p \sin\theta\sin\varphi + \p_z p \cos \varphi\big] =0.
\end{array}
\end{equation}
 We finally incorporate \eqref{DhatE3d}--\eqref{DhatE3d1}  to reformulate $H_3$ in \eqref{Lem3d:El2} by  incorporating the boundary condition   \eqref{bd3d:v} in (v) and (vi) to arrive at
\begin{equation}\begin{array}{cl}\label{Lem3d:Q3}
\hspace{-0.15in} \ds H_3 & \hspace{-0.125in} \ds =  \cos\theta \sin \varphi  \big[ v_3(v_2^2 + v_3^2 ) \p_x p -v_1 v_2 v_3 \p_y p + v_1 (v_1^2 + v_2^2 )  \p_z p \big] \\[0.1in]
\hspace{-0.15in}& \hspace{-0.125in} \ds \qquad  \ds  + \sin\theta\sin \varphi  \big[- v_1 v_2 v_3 \p_x p + v_3 ( v_1^2 + v_3^2) \p_y p + v_2
(v_1^2 + v_2^2 )\p_z p \big]\\[0.1in]
\hspace{-0.15in}& \hspace{-0.125in} \ds \qquad \ds  + (v_3 \cos \varphi)  \big[- v_1  v_3  \p_x p - v_2 v_3  \p_y p +  (2 v_1 ^2 + 2 v_2^2 + v_3^2 )\p_z p \big]\\[0.01in]
\hspace{-0.15in}& \hspace{-0.125in} \ds  \overset{\mbox{(v)}}= \cos\theta \sin \varphi \big[ v_3 |\bm v|^2 \p_x p - v_1 |\bm v|^2 \p_z p\big] + \sin\theta \sin\varphi\big[v_3 |\bm v|^2 \p_y p - v_2 |\bm v|^2 \p_z p \big] \\[0.01in]
\hspace{-0.15in}& \hspace{-0.125in} \ds  \overset{\mbox{(vi)}}= v_3 |\bm v|^2 \big[\p_x p \cos\theta\sin\varphi +   \p_y p \sin\theta\sin\varphi + \p_z p \cos \varphi\big] =0.
\end{array}
\end{equation}

We invoke \eqref{Lem3d:Q1}--\eqref{Lem3d:Q3} to find that $\big[(\bm E^2 \otimes \nabla p) \nabla c\big] \cdot \bm n(\bm x) =0$  on $ \p \om \times [0,T]$, which, combined with  \eqref{Lem3d:Et}, yields that $\big[(\bm E \otimes \nabla p) \nabla c \big] \cdot \bm n(\bm x) =0$  on $ \p \om \times [0,T]$, and we  thus complete the proof.
\end{proof}

We define a 3-by-3 block matrix $\hat{\bm E } =  \big(\hat{\bm E}_{i,j}\big)_{i,j=1}^3$ which exhibits a similar structure as its two-dimensional analogue \eqref{hatE}
{\footnotesize\begin{equation}\begin{array}{cl}\label{hatE3d}
   \hspace{-0.2in}\ds \hat{\bm E} \!:=\!
 \hspace{-0.05in}\begin{bmatrix}
  \big(\p_{v_1} D_{1,1},\p_{v_1} D_{1,2},\p_{v_1} D_{1,3} \big) &\hspace{-0.15in} \big( \p_{ v_1} D_{2,1}, \p_{ v_1} D_{2,2}, \p_{v_1} D_{2,3}\big) &\hspace{-0.15in} \big( \p_{ v_1} D_{3,1}, \p_{ v_1} D_{3,2}, \p_{ v_1} D_{3,3}\big)  \\[0.1in]
  \big(\p_{v_2} D_{1,1},\p_{v_2} D_{1,2},\p_{v_2} D_{1,3} \big) &\hspace{-0.15in} \big( \p_{ v_2} D_{2,1}, \p_{v_2} D_{2,2}, \p_{v_2} D_{2,3}\big) &\hspace{-0.15in} \big( \p_{ v_2} D_{3,1}, \p_{ v_2} D_{3,2},\p_{ v_2} D_{3,3}\big)  \\[0.1in]
    \big(\p_{v_3} D_{1,1},\p_{v_3} D_{1,2},\p_{v_3} D_{1,3} \big) &\hspace{-0.15in} \big( \p_{ v_3} D_{2,1}, \p_{ v_3} D_{2,2}, \p_{v_3} D_{2,3}\big) &\hspace{-0.15in} \big( \p_{ v_3} D_{3,1}, \p_{ v_3} D_{3,2}, \p_{ v_3} D_{3,3}\big)
\end{bmatrix}
\end{array}
\end{equation}}

\begin{lemma}\label{Lem3d2} The 3-by-3 block  matrices $\bm E =\big(\p_{\bm v}D_{i,j}\big)_{i,j=1}^3$  and $\hat {\bm E}$ defined in \eqref{hatE3d} satisfy the relation \eqref{Lem2:eq1} in Lemma \ref{Lem2}
%\begin{equation}\begin{array}{c}\label{Lem3d2:eq1}
%  \ds  \big[K(\bm x) (\bm E (\bm x, \bm v) \otimes  \nabla p )\nabla c\big] \cdot \nabla \psi_c = \nabla  p \cdot \big[K(\bm x) (\hat{\bm E} (\bm x, \bm v)\otimes  \nabla c) \nabla \psi_c \big]
%  \end{array}
%\end{equation}
in $\Omega \times [0,T]$.
\end{lemma}
\begin{proof}
The proof of this lemma follows from that of Lemma \ref{Lem2} and is thus omitted.
\end{proof}

\begin{lemma}\label{LemLPre3d}
For some scalar function $\psi_c$ and under the boundary conditions
\begin{equation}\label{Lem3dPre:bd}
 \bm v(\bm x, t) \cdot \bm n(\bm x)= (\bm D \nabla \psi_c) \cdot \bm n(\bm x) =0 ~~ \mbox{on}~~  \p \om \times [0,T],
\end{equation}
the following relations hold
\begin{equation}\begin{array}{l}\label{LemPre3d:eq1}
 \cos\theta\sin\varphi\big( (v_2^2 + v_3^2) \p_x \psi_c - v_1 v_2 \p_y \psi_c - v_1 v_3 \p_z \psi_c\big) \\[0.1in]
 \ds \qquad \qquad = (v_2 \cos \varphi -v_3 \sin\theta\sin\varphi)(v_3 \p_y \psi_c-v_2 \p_z \psi_c),\\[0.1in]
\ds    \sin\theta\sin\varphi\big(- v_1 v_2 \p_x\psi_c  + (v_1^2 + v_3^2) \p_y \psi_c - v_2 v_3 \p_z \psi_c\big) \\[0.1in]
 \ds \qquad\qquad = (v_1 \cos\varphi - v_3 \cos\theta\sin\varphi)(v_3 \p_x \psi_c - v_1 \p_z\psi_c),\\[0.1in]
\ds  \cos\varphi\big(- v_1 v_3 \p_x  \psi_c - v_2 v_3 \p_y  \psi_c  + (v_1^2 + v_2^2) \p_z \psi_c\big) \\[0.1in]
 \ds \qquad\qquad =(v_1 \sin\theta \sin\varphi-v_2 \cos\theta\sin\varphi)(v_2 \p_x \psi_c - v_1 \p_y \psi_c)
  \end{array}
\end{equation}
on $\p \Omega \times [0,T]$. Here the angles $\theta$ and $\varphi$ are determined by the outward unit normal vector $\bm n$ as described below \eqref{LemD3d:Dg}.
\end{lemma}
\begin{proof}
We prove the first relation in \eqref{LemPre3d:eq1} for demonstration, and the proofs of the remaining two equalities  are omitted due to similarity.

We utilize the following relation
\begin{equation}\label{bd3d:psic}\begin{array}{l}
\ds \p_x \psi_c \cos\theta\sin\varphi +   \p_y \psi_c \sin\theta\sin\varphi + \p_z \psi_c \cos \varphi=0 ~~\mbox{on}  ~~ \p \om \times [0,T]
\end{array}
\end{equation}
which is obtained from $(\bm D \nabla \psi_c) \cdot \bm n(\bm x) =0 $ on  $\p \om \times [0,T]$ by Lemma \ref{Lem3d:D},  to reformulate the left-hand side term in the first equality in \eqref{LemPre3d:eq1} as follows
\begin{equation}\begin{array}{l}\label{LemPre3d:eq2}
\hspace{-0.15in} \cos\theta\sin\varphi \big( (v_2^2 + v_3^2) \p_x \psi_c - v_1 v_2 \p_y \psi_c - v_1 v_3 \p_z \psi_c\big) \\[0.1in]
  = ( -\sin\theta\sin\varphi \p_y \psi_c) (v_2^2 + v_3^2) +  ( -\cos\varphi \p_z \psi_c) (v_2^2 + v_3^2)\\[0.1in]
   \ds \quad + \cos\theta\sin\varphi \big( - v_1 v_2 \p_y \psi_c - v_1 v_3 \p_z \psi_c\big)\\[0.1in]
  \ds =  -\p_y \psi_c \big[(v_2^2 + v_3^2) \sin\theta\sin\varphi + v_1 v_2  \cos\theta\sin\varphi \big]\\[0.1in]
   \ds \quad -\p_z \psi_c \big[(v_2^2 + v_3^2) \cos\varphi + v_1 v_3  \cos\theta\sin\varphi \big]\\[0.1in]
   \ds =  -\p_y \psi_c \big[ v_3^2 \sin\theta\sin\varphi - v_2 v_3  \cos\varphi \big]   -\p_z \psi_c \big[ v_2^2 \cos\varphi - v_2 v_3  \sin\theta\sin\varphi \big]\\[0.1in]
   \ds = (v_2 \cos \varphi -v_3 \sin\theta\sin\varphi)(v_3 \p_y \psi_c-v_2 \p_z \psi_c),
  \end{array}
\end{equation}
where we utilize the boundary condition \eqref{bd3d:v} to obtain the third equality, and  we thus prove the first equality in \eqref{LemPre3d:eq1}.
\end{proof}

\begin{lemma}\label{Lem3d3}  Under the boundary conditions \eqref{Lem3:bd} in Lemma \ref{Lem3}, we have $\ds \big[ (\hat{\bm E} (\bm x, \bm v) \otimes \nabla c ) \nabla \psi_c\big] \cdot \bm n(\bm x) =0$ on $ \p \om \times [0,T]$ with $\hat {\bm E}(\bm x, \bm v)$ defined  in \eqref{hatE3d}.
\end{lemma}
\begin{proof}
We combine \eqref{hatE3d} with the split in \eqref{DhatE3d} to accordingly split $  \hat {\bm E} := \hat {\bm  E^1}+ \hat  {\bm E^2}$  with $\hat{\bm E^1 } = \big(\hat{\bm E}^1_{i,j}\big)_{i,j=1}^3$ and $\hat{\bm E^2 } = \big(\hat{\bm E}^2_{i,j}\big)_{i,j=1}^3$  defined as follows
\begin{equation}\begin{array}{l}\label{Lem3d3:Matrix1}
\hspace{-0.15in}\ds
\hat{\bm E^1}: =\begin{bmatrix}
  \big([\bm E_d]_1,0,0\big) & \big(0, [\bm E_d]_1,0\big),  & \big(0, 0,[\bm E_d]_1\big)  \\[0.1in]
  \big([\bm E_d]_2,0,0\big)  &  \big(0, [\bm E_d]_2,0\big) & \big(0, 0,[\bm E_d]_2\big)\\[0.1in]
  \big([\bm E_d]_3,0,0\big)  &  \big(0, [\bm E_d]_3,0\big) & \big(0, 0,[\bm E_d]_3\big)
\end{bmatrix},
\end{array}
\end{equation}
\begin{equation}\begin{array}{l}\label{Leme3d3:Matrix2}
\hspace{-0.1in} {\footnotesize \hspace{-0.2in}\ds \hat{\bm E^2} \hspace{-0.02in} \ds : = \hspace{-0.05in} \ds
\begin{bmatrix}
  \big([\bm E_{l,1}]_1,\p_{v_1} D_{1,2},\p_{v_1} D_{1,3} \big) &\hspace{-0.15in} \big( \p_{ v_1} D_{2,1}, [\bm E_{l,2}]_1, \p_{v_1} D_{2,3}\big) &\hspace{-0.15in} \big( \p_{ v_1} D_{3,1}, \p_{ v_1} D_{3,2}, [\bm E_{l,3}]_1\big)  \\[0.1in]
  \big([\bm E_{l,1}]_2,\p_{v_2} D_{1,2},\p_{v_2} D_{1,3} \big) &\hspace{-0.15in} \big( \p_{ v_2} D_{2,1}, [\bm E_{l,2}]_2, \p_{v_2} D_{2,3}\big) &\hspace{-0.15in} \big( \p_{ v_2} D_{3,1}, \p_{ v_2} D_{3,2}, [\bm E_{l,3}]_2\big)  \\[0.1in]
    \big([\bm E_{l,1}]_3,\p_{v_3} D_{1,2},\p_{v_3} D_{1,3} \big) &\hspace{-0.15in} \big( \p_{ v_3} D_{2,1}, [\bm E_{l,2}]_3, \p_{v_3} D_{2,3}\big) &\hspace{-0.15in} \big( \p_{ v_3} D_{3,1}, \p_{ v_3} D_{3,2}, [\bm E_{l,3}]_3\big)
    \end{bmatrix}.
}
\end{array}
\end{equation}
Here $[\bm b]_i$ for $i =1, 2, 3$ refers to the $i$-th element of the vector  $\bm b \in \mathbb R^{3}$.

We first show that $\big[(\hat{\bm E^1} \otimes \nabla c) \nabla \psi_c\big] \cdot \bm n(\bm x) =0$ on  $\p \om \times [0,T]$.
 By  \eqref{Lem3d3:Matrix1}, we evaluate
\begin{equation}\begin{array}{cl}\label{Lem3d3:Et}
\hspace{-0.15in} \ds  \big[(\hat{\bm E^1} \otimes \nabla c) \nabla \psi_c\big] \cdot \bm n(\bm x) & \hspace{-0.1in}\ds =\f{1}{|\bm v|}
\begin{bmatrix}
v_1 \p_x c   & v_1 \p_y c  & v_1 \p_z c \\[0.05in]
v_2 \p_x c &v_2 \p_y c &v_2 \p_z c\\[0.05in]
v_3 \p_x c &v_3 \p_y c &v_3 \p_z c
\end{bmatrix}
\nabla \psi_c \cdot \bm n(\bm x)\\[0.3in]
& \hspace{-0.1in} \ds \!= \!\f{\big(\nabla c\cdot \nabla \psi_c \big)}{|\bm v|} \big[v_1 \cos\theta \sin \varphi  +v_2 \sin\theta\sin \varphi   +v_3 \cos \varphi \big] \\[0.1in]
& \hspace{-0.1in} \ds =0,
\end{array}
\end{equation}
in which the last equality is by the boundary condition \eqref{bd3d:v}.
We then evaluate $\big[(\hat {\bm E^2}\otimes \nabla c) \nabla \psi_c\big] \cdot \bm n(\bm x)$  on  $\p \om \times [0,T]$ by following the similar procedures in \eqref{Lem3d:El}--\eqref{Lem3d:El2} to obtain
\begin{equation}\begin{array}{l}\label{Lem3d3:El}
  \hspace{-0.3in} \ds\big[(\hat {\bm E^2}\otimes \nabla c) \nabla \psi_c\big] \cdot \bm n(\bm x)  \\[0.1in]
 \hspace{-0.1in}\ds = \p_x \psi_c \Big[ \cos\theta \sin\varphi \big( \hat {\bm E}^2_{1,1} \nabla c\big) + \sin\theta\sin\varphi  \big(\hat {\bm E}^2_{2,1}   \nabla c\big) + \cos \varphi \big(\hat {\bm E}^2_{3,1} \nabla c\big) \Big] \\[0.1in]
 \hspace{-0.1in}\ds \quad + \p_y \psi_c \Big[  \cos\theta \sin\varphi \big( \hat {\bm E}^2_{1,2} \nabla c\big) + \sin\theta\sin\varphi  \big(\hat {\bm E}^2_{2,2}   \nabla c\big) + \cos \varphi \big(\hat {\bm E}^2_{3,2} \nabla c\big) \Big] \\[0.1in]
 \hspace{-0.1in}\ds \quad + \p_z \psi_c \Big[\cos\theta \sin\varphi \big( \hat {\bm E}^2_{1,3} \nabla c\big) + \sin\theta\sin\varphi  \big(\hat {\bm E}^2_{2,3}   \nabla c\big) + \cos \varphi \big(\hat {\bm E}^2_{3,3} \nabla c\big)\Big] \\[0.1in]
 \hspace{-0.1in}\ds = :  \f{d_l-d_t}{|\bm v|^3} \big(G_1 \p_x \psi_c  + G_2 \p_y \psi_c +  G_3 \p_z \psi_c\big),
\end{array}
\end{equation}
in which $G_1$ could be evaluated by incorporating \eqref{DhatE3d}--\eqref{DhatE3d1} and \eqref{Leme3d3:Matrix2} as follows
\begin{equation}\begin{array}{cl}\label{Lem3d:G1}
\hspace{-0.15in}\ds G_1
& \hspace{-0.125in} \ds =  \cos\theta \sin \varphi  \big[ v_1 (v_1^2 + 2v_2^2 + 2v_3^2 ) \p_x c + v_2 (v_2^2 + v_3^2) \p_y c +  v_3 (v_2^2 + v_3^2) \p_z c \big] \\[0.1in]
\hspace{-0.15in}& \hspace{-0.125in} \ds \quad + \sin\theta\sin \varphi  \big[ -v_1^2 v_2 \p_x c + v_1(v_1^2 + v_3^2) \p_y c - v_1 v_2 v_3 \p_z c \big]\\[0.1in]
\hspace{-0.15in}& \hspace{-0.125in} \ds \quad + \cos \varphi  \big[ -v_1^2 v_3 \p_x c  - v_1 v_2 v_3 \p_y c + v_1 (v_1 ^2 + v_2^2 )\p_z c \big]\\[0.01in]
& \hspace{-0.125in} \ds  \overset{\mbox{(\rom{1})}} =  \cos\theta \sin \varphi  \big[  v_2 (v_2^2 + v_3^2) \p_y c +  v_3 (v_2^2 + v_3^2) \p_z c \big] \\[0.1in]
\hspace{-0.15in}& \hspace{-0.125in} \ds \quad + \sin\theta\sin \varphi  \big[ -v_1^2 v_2 \p_x c - v_1(2v_2^2 + v_3^2) \p_y c - v_1 v_2 v_3 \p_z c \big]\\[0.1in]
\hspace{-0.15in}& \hspace{-0.125in} \ds \quad + \cos \varphi  \big[ -v_1^2 v_3 \p_x c  - v_1 v_2 v_3 \p_y c - v_1 (v_2 ^2 + 2v_3^2 )\p_z c \big],
\end{array}
\end{equation}
where (\rom{1}) utilizes
\begin{equation}\label{bd3d:c}\begin{array}{l}
\ds \p_x c \cos\theta\sin\varphi +   \p_y c \sin\theta\sin\varphi + \p_z c \cos \varphi=0 ~~\mbox{on}  ~~ \p \om \times [0,T]
\end{array}
\end{equation}
derived from the boundary condition $(\bm D \nabla c) \cdot \bm n(\bm x) =0 $  by Lemma \ref{Lem3d:D} to cancel the term $\cos\theta \sin \varphi \p_x c$  on the right hand side of the first equality.
We evaluate  $G_2$ in \eqref{Lem3d3:El} by  \eqref{DhatE3d}--\eqref{DhatE3d1} and \eqref{Leme3d3:Matrix2}  in   a similar manner to obtain
\begin{equation}\begin{array}{cl}\label{Lem3d:G2}
\hspace{-0.15in}\ds G_2& \hspace{-0.125in} \ds =  \cos\theta \sin \varphi  \big[ v_2 (v_2^2 + v_3^2 ) \p_x c - v_1 v_2^2  \p_y c -v_1 v_2 v_3 \p_z c \big] \\[0.1in]
\hspace{-0.15in}& \hspace{-0.125in} \ds \quad + \sin\theta\sin \varphi  \big[ v_1(v_1^2 + v_3^2)  \p_x c + v_2(2v_1^2 + v_2^2 + 2v_3^2) \p_y c + v_3(v_1^2 + v_3^2) \p_z c \big]\\[0.1in]
\hspace{-0.15in}& \hspace{-0.125in} \ds \quad + \cos \varphi  \big[ -v_1 v_2 v_3 \p_x c  -  v_2^2 v_3 \p_y c + v_2 (v_1 ^2 + v_2^2 )\p_z c \big]\\[0.01in]
& \hspace{-0.175in} \ds  \overset{\mbox{(\rom{2})} } =  \cos\theta \sin \varphi  \big[ - v_1 v_2^2  \p_y c -v_1 v_2 v_3 \p_z c \big] \\[0.1in]
\hspace{-0.15in}& \hspace{-0.125in} \ds \quad + \sin\theta\sin \varphi  \big[ v_1(v_1^2 + v_3^2)  \p_x c + v_2(2v_1^2 + v_3^2) \p_y c + v_3(v_1^2 + v_3^2) \p_z c \big]\\[0.1in]
\hspace{-0.15in}& \hspace{-0.125in} \ds \quad + \cos \varphi  \big[ -v_1 v_2 v_3 \p_x c  -  v_2^2 v_3 \p_y c + v_2 (v_1 ^2 - v_3^2 )\p_z c \big],
\end{array}
\end{equation}
where (\rom{2}) similarly utilizes \eqref{bd3d:c} to cancel the term $\cos\theta \sin \varphi \p_x c$  on the right hand side of the first equality.
We finally evaluate   $G_3$ in \eqref{Lem3d3:El}  and incorporate \eqref{bd3d:c} in (\rom{3}) to arrive at
\begin{equation}\begin{array}{cl}\label{Lem3d:G3}
\hspace{-0.15in}\ds G_3
& \hspace{-0.125in} \ds =  \cos\theta \sin \varphi  \big[ v_3 (v_2^2 + v_3^2 ) \p_x c - v_1 v_2 v_3  \p_y c -v_1  v_3^2 \p_z c \big] \\[0.1in]
\hspace{-0.15in}& \hspace{-0.125in} \ds \quad + \sin\theta\sin \varphi  \big[ -v_1v_2 v_3  \p_x c + v_3(v_1^2  + v_3^2) \p_y c - v_2 v_3^2 \p_z c \big]\\[0.1in]
\hspace{-0.15in}& \hspace{-0.125in} \ds \quad + \cos \varphi  \big[ v_1 (v_1^2 + v_2^2 )  \p_x c  + v_2 (v_1^2 + v_2^2 ) \p_y c + v_3 (2v_1 ^2 + 2v_2^2 + v_3^2 )\p_z c \big]\\[0.01in]
& \hspace{-0.2in} \ds  \overset{\mbox{(\rom{3})} } =  \cos\theta \sin \varphi  \big[ - v_1 v_2 v_3  \p_y c -v_1  v_3^2 \p_z c \big] \\[0.1in]
\hspace{-0.15in}& \hspace{-0.125in} \ds \quad + \sin\theta\sin \varphi  \big[ -v_1v_2 v_3  \p_x c + v_3(v_1^2  - v_2^2) \p_y c - v_2 v_3^2 \p_z c \big]\\[0.1in]
\hspace{-0.15in}& \hspace{-0.125in} \ds \quad + \cos \varphi  \big[ v_1 (v_1^2 + v_2^2 )  \p_x c  + v_2 (v_1^2 + v_2^2 ) \p_y c + v_3 (2v_1 ^2 + v_2^2 )\p_z c \big].
\end{array}
\end{equation}
We then combine  \eqref{Lem3d:G1}--\eqref{Lem3d:G3} to rearrange  $\big[(\hat{\bm E^2}\otimes \nabla c) \nabla \psi_c\big]$ in \eqref{Lem3d3:El}  by collecting the terms with $ \p_x c v_1$, $ \p_y c v_2$, or $ \p_z c v_3$  to obtain
\begin{equation}\begin{array}{l}\label{Lem3d3:Ele2}
\hspace{-0.25in}  \ds \f{|\bm v|^3}{d_l-d_t} \big[(\hat{\bm E^2}\otimes \nabla c) \nabla \psi_c\big] \cdot \bm n(\bm x)  \\[0.15in]
\hspace{-0.15in} \ds =  \big(\p_y c v_2 + \p_z c v_3 \big) \cos\theta\sin\varphi\big[ (v_2^2 + v_3^2) \p_x \psi_c - v_1 v_2 \p_y \psi_c - v_1 v_3 \p_z \psi_c\big] \\[0.15in]
\ds   + \big(\p_x c v_1 + \p_z c v_3 \big) \sin\theta\sin\varphi\big[- v_1 v_2 \p_x\psi_c  + (v_1^2 + v_3^2) \p_y \psi_c - v_2 v_3 \p_z \psi_c\big] \\[0.15in]
\ds  + \big(\p_x c v_1 + \p_y c v_2 \big) \cos\varphi\big[- v_1 v_3 \p_x  \psi_c - v_2 v_3 \p_y  \psi_c  + (v_1^2 + v_2^2) \p_z \psi_c\big] \\[0.15in]
\ds + \sin\theta\sin\varphi \p_y c\big[ -v_1 (v_3^2 + 2 v_2^2)\p_x \psi_c \!+\! v_2 (2 v_1^2 + v_3^2)\p_y \psi_c \!+\! v_3 (v_1^2- v_2^2)\p_z \psi_c\big]\\[0.15in]
\ds + \cos\varphi \p_z c\big[ -v_1 (v_2^2 + 2 v_3^2)\p_x \psi_c + v_2 ( v_1^2 - v_3^2)\p_y \psi_c + v_3 (2v_1^2+ v_2^2)\p_z \psi_c\big].
\end{array}
\end{equation}
By splitting the fourth term on the right hand side of \eqref{Lem3d3:Ele2}, we find that
\begin{equation}\begin{array}{l}\label{Lem3d3:Ele3}
\hspace{-0.2in}  \ds  \sin\theta\sin\varphi \p_y c\big[ -v_1 (v_3^2 + 2 v_2^2)\p_x \psi_c + v_2 (2 v_1^2 + v_3^2)\p_y \psi_c + v_3 (v_1^2- v_2^2)\p_z \psi_c\big]\\[0.125in]
\ds = \ds  \sin\theta\sin\varphi (\p_y c v_2 )\big[- v_1 v_2 \p_x\psi_c  + (v_1^2 + v_3^2) \p_y \psi_c - v_2 v_3 \p_z \psi_c\big]\\[0.125in]
\ds \quad + \sin\theta\sin\varphi (\p_y c  v_1)\big[ -(v_2^2 + v_3^2) \p_x \psi_c + v_1 v_2 \p_y \psi_c + v_1 v_3 \p_z \psi_c\big],
\end{array}
\end{equation}
where the last term in \eqref{Lem3d3:Ele3} could be further simplified by \eqref{bd3d:c} as follows
\begin{equation}\begin{array}{l}\label{Lem3d3:Ele5}
\hspace{-0.2in}   \sin\theta\sin\varphi \p_y c  v_1\big[ -(v_2^2 + v_3^2) \p_x \psi_c + v_1 v_2 \p_y \psi_c + v_1 v_3 \p_z \psi_c\big]\\[0.1in]
\ds = \ds \cos\theta\sin\varphi (\p_x cv_1)\big[ (v_2^2 + v_3^2) \p_x \psi_c - v_1 v_2 \p_y \psi_c - v_1 v_3 \p_z \psi_c\big]\\[0.1in]
\ds \quad +  \cos\varphi (\p_z c v_1)\big[ (v_2^2 + v_3^2) \p_x \psi_c - v_1 v_2 \p_y \psi_c - v_1 v_3 \p_z \psi_c\big].
\end{array}
\end{equation}
We similarly split the fifth term on the right hand side of \eqref{Lem3d3:Ele2} as follows
\begin{equation}\begin{array}{l}\label{Lem3d3:Ele4}
\hspace{-0.2in} \ds \cos\varphi \p_z c\big[ -v_1 (v_2^2 + 2 v_3^2)\p_x \psi_c + v_2 ( v_1^2 - v_3^2)\p_y \psi_c + v_3 (2v_1^2+ v_2^2)\p_z \psi_c\big]\\[0.125in]
\ds =\cos\varphi (\p_z c v_3)\big[- v_1 v_3 \p_x  \psi_c - v_2 v_3 \p_y  \psi_c  + (v_1^2 + v_2^2) \p_z \psi_c\big] \\[0.125in]
 \ds \quad - \cos\varphi (\p_z c v_1) \big[ (v_2^2 + v_3^2) \p_x \psi_c - v_1 v_2 \p_y \psi_c - v_1 v_3 \p_z \psi_c\big].
\end{array}
\end{equation}

We invoke \eqref{Lem3d3:Ele3}--\eqref{Lem3d3:Ele4} into \eqref{Lem3d3:Ele2} and  cancel the like terms to reformulate \eqref{Lem3d3:Ele2} as follows
\begin{equation}\begin{array}{l}\label{Lem3d3:Ele6}
\hspace{-0.2in} \ds \f{|\bm v|^3}{d_l-d_t} \big[(\hat{\bm E^2}\otimes \nabla c) \nabla \psi_c\big] \cdot \bm n(\bm x)  \\[0.15in]
  \hspace{-0.1in}\ds  \ds =  \big(\nabla  c \cdot \bm v \big) \Big[ \cos\theta\sin\varphi\big( (v_2^2 + v_3^2) \p_x \psi_c - v_1 v_2 \p_y \psi_c - v_1 v_3 \p_z \psi_c\big) \\[0.1in]
  \hspace{-0.1in}\ds  \quad\ds  +  \sin\theta\sin\varphi\big(- v_1 v_2 \p_x\psi_c  + (v_1^2 + v_3^2) \p_y \psi_c - v_2 v_3 \p_z \psi_c\big) \\[0.1in]
  \hspace{-0.1in}\ds \quad \ds  + \cos\varphi\big(- v_1 v_3 \p_x  \psi_c - v_2 v_3 \p_y  \psi_c  + (v_1^2 + v_2^2) \p_z \psi_c\big) \Big] \\[0.1in]
  \hspace{-0.1in}\ds \ds  = \big(\nabla  c \cdot \bm v \big) \Big[ (v_2 \cos \varphi -v_3 \sin\theta\sin\varphi)(v_3 \p_y \psi_c-v_2 \p_z \psi_c)\\[0.1in]
  \hspace{-0.1in}\ds \quad \ds   +  (v_1 \cos\varphi - v_3 \cos\theta\sin\varphi)(v_3 \p_x \psi_c - v_1 \p_z\psi_c)\\[0.1in]
  \hspace{-0.1in}\ds  \quad \ds   + (v_1 \sin\theta \sin\varphi-v_2 \cos\theta\sin\varphi)(v_2 \p_x \psi_c - v_1 \p_y \psi_c) \Big]= : \big(\nabla  c \cdot \bm v \big) \hat J,
\end{array}
\end{equation}
where the second equality is by Lemma \ref{LemLPre3d}. We evaluate $\hat J := \hat J_1 + \hat J_2$ in \eqref{Lem3d3:Ele6} to obtain
 \begin{equation}\begin{array}{l}\label{Lem3d3:hatQ1}
\ds \hat J_1 \ds := -v_1^2 \big[ \p_y \psi_c \sin\theta\sin\varphi + \p_z \psi_c \cos \varphi\big] \ds -v_2^2 \big[\p_x \psi_c \cos\theta\sin\varphi + \p_z \psi_c \cos \varphi \big]    \\[0.1in]
\ds \qquad \quad  \ds  -v_3^2 \big[ \p_x \psi_c \cos\theta\sin\varphi + \p_y \psi_c \sin\theta\sin\varphi)\big]\\[0.1in]
\ds \qquad  = v_1^2\p_x \psi_c \cos\theta\sin\varphi + v_2^2 \p_y \psi_c  \sin\theta\sin\varphi + v_3^2 \p_z \psi_c \cos \varphi
  \end{array}
\end{equation}
by \eqref{bd3d:psic} and
 \begin{equation}\begin{array}{l}\label{Lem3d3:hatQ2}
\ds \hat J_2 \ds := v_1 \p_x \psi_c  \big[ v_2 \sin\theta\sin\varphi + v_3 \cos \varphi\big]  + v_2 \p_y \psi_c \big[v_1 \cos\theta\sin\varphi + v_3 \cos \varphi \big]    \\[0.1in]
\ds \qquad \quad  +  v_3 \p_z \psi_c \big[ v_1 \cos\theta\sin\varphi + v_2 \sin\theta\sin\varphi)\big] \\[0.1in]
\ds \qquad \ds = -v_1^2\p_x \psi_c \cos\theta\sin\varphi - v_2^2 \p_y \psi_c  \sin\theta\sin\varphi - v_3^2 \p_z \psi_c \cos \varphi
  \end{array}
\end{equation}
by \eqref{bd3d:v}. By \eqref{Lem3d3:hatQ1}--\eqref{Lem3d3:hatQ2}, we thus have $\hat J=0$ in \eqref{Lem3d3:Ele6}, which, combined with \eqref{Lem3d3:Et}, completes the proof.
\end{proof}

\begin{remark}\label{Opt:3d}
  With the assistance of the complex boundary conditions in \S \ref{Sect:Lem3d}, we could follow the proof of Theorem \ref{thm:OptCond} to  accordingly  derive the first-order optimality condition for the three-dimensional optimal control problem \eqref{Model}--\eqref{ObjFun}.
\end{remark}
\section*{Acknowledgments}
This work was partially supported by the National Natural Science Foundation of China (No. 12301555), the National Key R\&D Program of China (No. 2023YFA1008903),  the Taishan Scholars Program of Shandong Province (No. tsqn202306083),  the National Science Foundation under Grants DMS-2012291 and DMS-2245097.

	\end{document}